\documentclass[11pt]{article}

\usepackage{amsmath}
\usepackage{amssymb}
\usepackage{amsthm} %For Box at \end{proof}
\usepackage{mathrsfs}
\usepackage[dvipdfmx]{graphicx,color}
\nopagebreak[3]

\newcommand{\newsection}[1]
{\section{#1}\setcounter{theorem}{0} \setcounter{equation}{0} \par\noindent}

\newtheorem{theorem}{Theorem}

\newtheorem{lemma}[theorem]{Lemma}

\newtheorem{remark}[theorem]{Remark}

\newcommand{\beq}{ \begin{equation} }
\newcommand{\eeq}{ \end{equation} }

\newcommand{\br}{{\mathbb R}}
\newcommand{\bc}{{\mathbb C}}

\newcommand{\diag}{\mbox{diag}\ }

\newcommand{\brn}{ { \mathbb{R}^n } }
\newcommand{\two}{ { I \hspace{-3pt}I } }
\newcommand{\three}{ { I \hspace{-3pt}I \hspace{-3pt}I } }
\newcommand{\four}{ { I \hspace{-3pt}V } }

\newcommand{\anglex}[1]{\langle #1\rangle}
\newcommand{\real}{ \mbox{\rm{Re}\,} } 
\newcommand{\imaginary}{ \mbox{\rm{Im}\,} } 

\title{
The Cauchy problem 
\\
for a semilinear ordinary differential equation 
\\
in the homogeneous and isotropic spacetime
}
\author
{
Makoto NAKAMURA
\thanks
{
{Faculty of Science, Yamagata University, 
Kojirakawa-machi 1-4-12, Yamagata 990-8560, JAPAN.}
E-mail: \texttt{nakamura@sci.kj.yamagata-u.ac.jp}
}
}
\date{}

\begin{document}

\maketitle

\begin{abstract}
A semilinear ordinary differential equation is derived from a semilinear Schr\"odinger equation 
in the homogeneous and isotropic spacetime by the Ehrenfest theorem. 
The Cauchy problem for the equation is considered.
Exact solutions and nonexistence of global weak solutions of the equation 
are also considered in the de Sitter spacetime.
The effects of spatial expansion and contraction are studied.
\end{abstract}

\noindent
{\it Mathematics Subject Classification (2010)}: 
Primary 34A12; Secondary 34A34, 83C10. \\
%34A12 Initial value problems, existence, uniqueness, continuous dependence and continuation of solutions
%34A34 Nonlinear equations and systems, general
%83C10 Equations of motion
%35G20 Nonlinear higher-order equations
%35Q76 Einstein equations

\noindent
{\it Keywords}: 
ordinary differential equation,  
Cauchy problem,
homogeneous and isotropic spacetime. 

\newsection{Introduction}
The Cauchy problems for some semilinear Klein-Gordon equations have been considered 
in the de Sitter spacetime 
in \cite{Baskin-2013-AHP, Nakamura-2014-JMAA, Yagdjian-2012-JMAA}.
In some cases, the spatial expansion [resp. contraction] can be characterized as a dissipative [resp. anti-dissipative] term in the equations 
(for example, see the energy estimate (2.3) in \cite{Nakamura-2014-JMAA}), 
by which the existence-time of the solution is deeply affected.
A semilinear Schr\"odinger equation was derived from the semilinear Klein-Gordon equation 
in \cite{Nakamura-2015-JDE} by the nonrelativistic limit, and its Cauchy problem was considered.
In this paper, we derive a semilinear ordinary differential equation from the semilinear Schr\"odinger equation 
based on the Ehrenfest theorem, 
and we consider its Cauchy problem in more general spacetime known in Cosmology.

We consider the Cauchy problem for the semilinear ordinary differential equation given by 
\eqref{Cauchy-Y}, below, in the homogeneous and isotropic spacetime 
which is described by a scale-function given by 
\eqref{a}, below.
We start from the introduction  of the scale-function.
Let $n\ge1$, $\sigma\in \br$.
For $a_0>0$ and $a_1\in \br$, 
put $H:=a_1/a_0$  
which is called the Hubble constant.
When $\sigma\neq -1$ and $a_1\neq0$, we put 
\[
T_0:=-\frac{2a_0}{n(1+\sigma)a_1}.
\]
We define 
\[
T_1:=
\begin{cases}
T_0 & \mbox{if}\ \ (1+\sigma)a_1<0\ (\mbox{i.e., }T_0>0), 
\\
\infty & \mbox{if}\ \ (1+\sigma)a_1\ge0.
\end{cases}
\]
Under the assumption of the cosmological principle, 
namely, that the space is homogeneous and isotropic, 
the solution of the Einstein equations 
with the spatial flat curvature is given by 
the Friedmann-Lema\^{i}tre-Robertson-Walker metric (the FLRW metric) given by 
\begin{equation}
\label{Intro-RW}
-c^2(d\tau)^2
=
g_{\alpha\beta}dx^\alpha dx^\beta
=
-c^2(dx^0)^2+
a(x^0)^2
\sum_{j=1}^n (dx^j)^2,
\end{equation}
where 
$\tau$ is the proper time, 
and 
$a(\cdot)$ is the scale-function of the spacetime  
defined by 
\beq
\label{a}
a(t):=
\begin{cases}
a_0\left(1+\frac{n(1+\sigma)a_1t}{2a_0}\right)^{2/n(1+\sigma)} 
& \mbox{if}\ \sigma\neq -1, 
\\
a_0e^{a_1t/a_0} 
& \mbox{if}\ \sigma= -1
\end{cases}
\eeq
for $0\le t<T_1$ and $\sigma\in \br$ 
(see 
\cite{Carroll-2004-Addison, DInverno-1992-Oxford}
for the references of the Einstein equations and the FLRW metric).
We denote the first order derivative by the variable $t$ by 
\[
D_t:=\frac{d}{dt}.
\]
We note that $a(0)=a_0$ and $D_ta(0)=a_1$ hold.

The Minkowski spacetime is obtained when $a_1=0$ (i.e., $a(\cdot)=a_0$ is a constant), 
and the de Sitter spacetime is obtained when $\sigma=-1$. 
The scale-function $a(\cdot)$ defined by \eqref{a} blows up [resp. vanishes] at finite time $T_0>0$ 
when $a_1>0$ and $\sigma<-1$ [resp. $a_1<0$ and $\sigma>-1$], 
which is called Big-Rip [resp. Big-Crunch] in cosmology. 
The case $\sigma= -1$ shows the exponential expansion [resp. contraction]
of the space when $a_1>0$ [resp. $a_1<0$].
The case $a_1>0$ and $\sigma> -1$ 
[resp. $a_1<0$ and $\sigma< -1$] shows the polynomial expansion [resp. contraction] 
of the space. 
These models have been studied as some models of the universe.

We define a function $q_0(\cdot)$ and a weight function $A(\cdot)$ by 
\beq
\label{A}
q_0(t):=\frac{D_t a^2(t)}{a^2(t)}, 
\ \ 
A(t):=-\frac{1}{4}q_0(t)^2-\frac{1}{2} D_t q_0(t)
\eeq
for $0\le t<T_1$.
For $T$ with $0<T\le T_1$, 
we consider the Cauchy problem for a semilinear ordinary differential equation given by 
\beq
\label{Cauchy-Y}
\left\{
\begin{array}{l}
\displaystyle D_t^2 Y(t)+A(t) Y(t)+f(Y)(t)=0,
\\
Y(0)=Y_0,\ \ D_t Y(0)=Y_1 
\end{array}
\right.
\eeq
for $0\le t< T$, 
where 
$Y=(Y^1,\cdots,Y^n)$, 
$Y_0=(Y^1_0,\cdots,Y^n_0)\in \brn$, 
$Y_1=(Y^1_1,\cdots,Y^n_1)\in \brn$, 
\beq
\label{Def-Modulo}
|Y|:=\left\{\sum_{j=1}^n (Y^j)^2\right\}^{1/2},
\eeq
$\lambda\in \br$, $1<p<\infty$ 
and 
\beq
\label{Def-f(Y)}
f(Y):=\lambda |Y|^{p-1}Y\ \  \mbox{or}\ \  f(Y):=\lambda |Y|^p.
\eeq
The differential equation in \eqref{Cauchy-Y} is the Newtonian equation of motion  
and is derived from the semilinear Schr\"odinger equation 
in the homogeneous and isotropic spacetime  
with the scale-function $a(\cdot)$ in \eqref{a} by the Ehrenfest theorem 
(see Section \ref{Section-Derivation}, below).
We regard the solution of the Cauchy problem 
\eqref{Cauchy-Y} as the fixed point of the integral operator $\Psi$ defined by 
\eqref{Def-Psi}, below.

The differential equation in \eqref{Cauchy-Y} with $f(Y):=\lambda |Y|^{p-1}Y$ is rewritten as  
\beq
\label{Eq-Z}
D_s^2 Z+\frac{2\beta+1}{s} D_s Z+\frac{\beta^2+A(t(s))}{s^2} Z+\lambda s^{\beta(p-1)-2} |Z|^{p-1}Z=0
\eeq
by the change of variable $s:=e^t$ and $Z(s)=Y(t(s)) s^{-\beta}$ for $\beta\in \br$.
The equation 
$D_s^2Z+B(s) |Z|^{p-1}Z=0$ is known as the Emden-Fowler type equation for some function $B(\cdot)$. 
We refer to 
\cite{Kwong-Wong-2006-NA} for the nonoscillation, 
\cite{Li-2006-NA} for the blow-up rate,
\cite{Astashova-2019-DiffEq, 
Knezhevich-2007-DiffEq,
Mikic-2016-KragujevacJMath} 
for the asymptotic behavior, 
and 
\cite{Knezhevich-2009-DiffEq,Krtinic-Mikic-2019-MathNotes} 
for the uniqueness of the solution of the Emden-Fowler type equation.
The equation $D_s^2 Z+BZ/s^2-|Z|^{p-1}Z=0$ is considered in 
\cite{Astashova-2009-FunctDiffEq} for a constant $B$.
Our equation \eqref{Eq-Z} has a dissipative or anti-dissipative term 
$(2\beta+1)D_s Z/s$,  
and the coefficient of $Z/s^2$ is a function $\beta^2+A(t(s))$ 
dependent on the scale-function in \eqref{a}.

We say that $Y$ is a global solution of \eqref{Cauchy-Y} 
if $Y$ exists on the time-interval $[0,T_1)$  
since the spacetime does not exist after $T_1$.
Let us consider the four cases (i), (ii), (iii) and (iv) given by
\beq
\label{i-ii-iii}
\begin{array}{lll}
{\rm (i)} & a_1>0, & \sigma>-1+\frac{2}{n}, 
\\
{\rm (ii)} & a_1=0, & \sigma\in \br, 
\\
{\rm (iii)} & a_1\neq0, & \sigma=-1+\frac{2}{n},
\\
{\rm (iv)} & a_1<0, & \sigma>-1+\frac{2}{n}.
\end{array} 
\eeq
The scale-function $a(t)=a_0+a_1t$ given in the case (iii) 
corresponds to the scale-function for the Milne universe.
In the cases (i), (ii) and (iii), we have 
$A\ge0$ and $D_t A\le 0$ on $[0,T_1)$ 
(see Lemma \ref{Lem-9}, below).
We put 
\[
D:=\sqrt{A(0)}\, |Y_0|+|Y_1|.
\]

Let us consider the case (i).
For $T>0$ and $R>0$, we define the norm $\|\cdot\|_{X(T)}$ by 
\beq
\label{Def-X-Norm}
\|Y\|_{X(T)}:=\|D_t Y\|_{L^\infty((0,T))}
+
\|\sqrt{A}\, Y\|_{L^\infty((0,T))}
+
\|\sqrt{-D_t A}\, Y\|_{L^2((0,T))},
\eeq
and the function spaces $X(T)$ and $X(T,R)$ by 
\[
X(T):=\{Y;\ \|Y\|_{X(T)}<\infty\}
\ \ \mbox{and}\ \ 
X(T,R):=\{Y;\ \|Y\|_{X(T)}\le R\}.
\]

We obtain the local and global solutions of the problem \eqref{Cauchy-Y} as follows.

\begin{theorem}
[Cauchy problem for the case (i)]
\label{Thm-1}
Let $a_1>0$ and $\sigma>-1+2/n$ (so that, $T_0<0$, $T_1=\infty$).
Let $\lambda\in \br$ and $1<p<\infty$.
Let $q_\ast$ be an arbitrary number which satisfies $1\le q_\ast\le \infty$ and $1-p/2\le 1/q_\ast$.

(1)
For any $Y_0\in \brn$ and $Y_1\in \brn$, 
there exists $T>0$ and $R>0$ 
such that the Cauchy problem \eqref{Cauchy-Y} has 
a unique solution $Y\in X(T,R)$, 
where $T>0$ can be arbitrarily taken as 
\beq
\label{Thm-1-1000}
T\le 
\begin{cases}
-T_0\left\{
\left(
1-\frac{C}{T_0
D^{(p-1)q_\ast}
}
\right) 
^{1/(p+1)q_\ast}-1
\right\}
& \mbox{if}\ q_\ast<\infty, \ \ 0\le D<\infty, 
\\
-T_0
\left(
\frac{ C }{D^{(p-1)/(p+1)} } -1
\right) 
& \mbox{if}\ q_\ast=\infty,\ \ D< C^{(p+1)/(p-1)}
\end{cases}
\eeq
for some constant $C>0$ which is independent of $Y_0$ and $Y_1$.

(2) 
The solution $Y$ in (1) satisfies $Y\in C^1([0,T))$.

(3) 
The solution $Y$ in (1) is unique in $C^1([0,T))\cap X(T)$.

(4) 
If $\lambda\ge0$ and $f(Y)=\lambda |Y|^{p-1}Y$, 
then the solution $Y$ in (1) is a global solution.
Namely, $T$ can be taken as $T=T_1$.
\end{theorem}

Let us consider the cases (ii) and (iii).
For $T>0$, $R_0>0$ and $R_1>0$, 
we put
\[
X(T,R_0,R_1):=
\left\{
Y;\ \|Y\|_{L^\infty((0,T))}\le R_0,
\ \|D_t Y\|_{L^\infty((0,T))}\le R_1
\right\}.
\]

\begin{theorem}
[Cauchy problem for the cases (ii) and (iii)]
\label{Thm-2}
Let $a_1=0$ and $\sigma\in \br$ (so that, $T_1=\infty$), 
or $a_1\neq0$ and $\sigma=-1+2/n$ 
(so that, $T_1=\infty$ when $a_1>0$, 
and $T_1=T_0>0$ when $a_1<0$).
Let $\lambda\in \br$ and $1<p<\infty$.
Then the following results hold. 
Moreover, the results from (2) to (4) in Theorem \ref{Thm-1} hold 
with $X(T)$ replaced by 
\[
\left\{
Y;\ \|Y\|_{L^\infty((0,T))}<\infty,\ \ \|D_t Y\|_{L^\infty((0,T))}<\infty
\right\}.
\]

(1)
For any $Y_0\in \brn$ and $Y_1\in \brn$, 
there exists $T>0$, constants $C_0>0$ and $C>0$   
such that the Cauchy problem \eqref{Cauchy-Y} has 
a unique solution $Y\in X(T,R_0, R_1)$ 
for any $R_0$ and $R_1$ with $R_0\ge 2|Y_0|$ and $R_1\ge C_0|Y_1|$,  
where $T$ can be arbitrarily taken as 
\beq
\label{Thm-2-T}
0<T\le 
\min\left\{
\frac{C}{ R_0^{p-1} },\ \frac{C}{ R_0^{(p-1)/2} },\ \frac{CR_1}{R_0^p},\ \frac{R_0}{2R_1}
\right\}.
\eeq
Here, the constants $C_0$ and $C$ are independent of $Y_0$ and $Y_1$.

(2) 
When $\sigma=-1+2/n$ and $a_1<0$ (so that, $T_1=T_0>0$), 
if $|Y_0|$ and $|Y_1|$ are sufficiently small,
then the solution $Y$ in (1) is a global solution.
\end{theorem}

In (2) in Theorem \ref{Thm-2}, 
we obtain global solutions for small data 
when the space is contracting (i.e., $a_1<0$), 
while we have only local solutions in (1) in the Minkowski spacetime (i.e., $a_1=0$).

\vspace{10pt}

Let us consider the case (iv).
We note $T_1=T_0>0$, $A>0$ and $D_tA>0$ on $[0,T_1)$ in this case 
(see (7) in Lemma \ref{Lem-9}, below).
We put 
\begin{eqnarray}
\|Y\|_{X'(T)}
&:=& \|A^{-1/2}D_tY\|_{L^\infty((0,T))}
+\|Y\|_{L^\infty((0,T))}+\|A^{-1}\sqrt{D_tA} D_t Y\|_{L^2((0,T))},
\nonumber\\
X'(T)
&:=&\{Y;\ \|Y\|_{X'(T)}<\infty\},
\nonumber\\
X'(T,R)
&:=&\{Y;\ \|Y\|_{X'(T)}\le R\},
\nonumber\\
D'
&:=&|Y_0|+A(0)^{-1/2} |Y_1|
\label{Def-D'}
\end{eqnarray}
for $0<T\le T_1$ and $R>0$.

\begin{theorem}
[Cauchy problem for the case (iv)]
\label{Thm-3}
Let $a_1<0$ and $\sigma>-1+2/n$ (so that, $T_1=T_0>0$).
Let $\lambda\in \br$ and $1<p<\infty$.
Then there exists constants $C_0>0$ and $C>0$ such that the following results hold.

(1)
For any $Y_0\in \brn$ and $Y_1\in \brn$, 
there exists $T$ with $0<T< T_1$, 
constants $C_0>0$ and $C>0$  
such that the Cauchy problem \eqref{Cauchy-Y} has 
a unique solution $Y\in X'(T,C_0D')$, 
where $T$ can be arbitrarily taken under the condition 
\beq
\label{Thm-3-1000}
\frac{C|\lambda|a_0 T}{|a_1|}\cdot \left(1-\frac{T}{2T_0}\right) \cdot (C_0D')^{p-1}\le 1.
\eeq
Here, the constants $C_0$ and $C$ are independent of $Y_0$ and $Y_1$.

(2) 
If $D'$ is sufficiently small such that 
\beq
\label{Thm-3-2000}
\frac{2C|\lambda|a_0^2}{n(1+\sigma)|a_1|^2} \cdot (C_0 D')^{p-1}\le 1,
\eeq
then the solution $Y$ in (1) is a global solution.
Namely, $T$ can be taken as $T=T_1$.

(3) 
The solution $Y$ in (1) satisfies $Y\in C^1([0,T))$.

(4) 
The solution $Y$ in (1) is unique in $C^1([0,T))\cap X'(T)$.

(5) 
If $\lambda\ge0$ and $f(Y)=\lambda |Y|^{p-1}Y$, 
then the solution $Y$ in (1) is a global solution.
\end{theorem}

In (2) in Theorem \ref{Thm-3}, 
we obtain global solutions  
when the space is contracting (i.e., $a_1<0$) 
under the condition \eqref{Thm-3-2000}.
We note that global solutions are obtained even for large data 
(i.e., $D'$ is large) 
if the spatial contraction is sufficiently large (i.e., $|a_1|$ is sufficiently large).

\vspace{10pt}

In the above theorems, 
the case $\sigma\ge -1+2/n$ has been considered when $a_1\neq0$, 
which is required for our energy estimates 
(see Lemmas \ref{Lem-10} and \ref{Lem-12}, below).
When $a_1=0$, 
we note that $a(\cdot)=a_0$ in \eqref{a} is independent of $\sigma$.
The case $\sigma=-1$ in \eqref{a} corresponds to the de Sitter spacetime with the flat spatial curvature 
which is one of important models of the expanding or contracting universe.
In this case, we have the following two theorems.
The first theorem shows the exact solutions of the Cauchy problem \eqref{Cauchy-Y}.

\begin{theorem}
[Exact solutions in the de Sitter spacetime]
\label{Thm-4}
Let $n\ge1$, $H\in \br$, $a(t)=e^{Ht}$, $\lambda \in \br$, $p\in \br$, $f(Y)=\lambda |Y|^{p-1}Y$.

(1) 
Let $p=1$. 
Then the solution of \eqref{Cauchy-Y} is given by 
\beq
\label{Thm-4-1000}
Y(t)=
\begin{cases}
Y_0\cos \sqrt{\lambda-H^2} t+\frac{Y_1}{ \sqrt{\lambda-H^2} } \sin\sqrt{\lambda-H^2}t & \mbox{if}\ \lambda>H^2,
\\
Y_0+Y_1 t & \mbox{if}\ \lambda=H^2,
\\
B e^{\sqrt{H^2-\lambda} t}
+
C e^{-\sqrt{H^2-\lambda}t} & \mbox{if}\ \lambda<H^2,
\end{cases}
\eeq 
where $B$ and $C$ are constants defined by 
\beq
\label{Thm-4-2000}
B
:=
\frac{1}{2}
\left(Y_0+\frac{Y_1}{\sqrt{H^2-\lambda}}\right)
\ \ \mbox{and} \ \ 
C:=
\frac{1}{2}
\left(Y_0-\frac{Y_1}{\sqrt{H^2-\lambda}}\right).
\eeq

(2) 
Let $p\neq 1$.  
Let $Y^3=\cdots=Y^n=0$.
Put $R:=\sqrt{(Y_0^1)^2+(Y_0^2)^2}$.
Then the solution $Y$ of \eqref{Cauchy-Y} is given as follows.

(i) 
If $\lambda R^{p-1}>H^2$ 
(namely, $\lambda>0$ and $R>(H^2/\lambda)^{1/(p-1)}$), 
then 
\beq
\label{Thm-4-3000}
Y^1(t)=R\cos (\omega t+\delta),\ \ Y^2(t)=R\sin (\omega t+\delta),
\eeq
where we have put $\omega:=\sqrt{\lambda R^{p-1}-H^2}$ and $\delta$ is a number which satisfies 
$Y^1_0=R\cos \delta$, 
$Y^2_0=R\sin \delta$, 
$Y^1_1=-R\omega \sin \delta$ and $Y^2_1=R\omega \cos \delta$.

(ii) If $\lambda R^{p-1}=H^2$ 
(namely, $\lambda=H=0$, or $\lambda>0$ and $R=\left(H^2/{\lambda}\right)^{1/(p-1)}$), 
then 
\[
Y^1(t)=Y^1_0,\ \ Y^2(t)=Y^2_0 
\]
and $Y^1_1=Y^2_1=0$.
Namely, $Y^1$ and $Y^2$ are constants.

(iii) 
If $\lambda R^{p-1}<H^2$ 
(namely, $\lambda<0$ and $H\in \br$, 
or 
$\lambda=0$ and $H\neq 0$, 
or 
$\lambda>0$ and $R<(H^2/\lambda)^{1/(p-1)}$), 
then the solution $Y^1=Y^2=0$ is only allowed.
\end{theorem}

In Theorem \ref{Thm-4}, $\lambda=H^2$ and $\lambda R^{p-1}=H^2$ are thresholds for the results in 
(1) and (2), respectively.
Especially, we have the new solutions for the positive $\lambda>0$ with $0<\lambda<H^2$ in (1) and 
$0<\lambda<H^2/R^{p-1}$ for (iii) in (2) when $H\neq0$.

\begin{remark}
[Rotational motion by the gravity] 
\label{Thm-4-(3)-(4)}
Let $M>0$ be a mass constant. 
Let $G$ be the Newton gravitational constant.
The first equation in the Cauchy problem \eqref{Cauchy-Y} 
for 
$\sigma=-1$, $f(Y)=\lambda |Y|^{p-1}Y$, 
$p=-2$ and $\lambda =GM$ is written as 
\beq
\label{Eq-Rotation}
D_t^2Y(t)-H^2Y(t)+\frac{GM}{|Y(t)|^2}\cdot \frac{Y(t)}{|Y(t)|}=0
\eeq
for $0\le t<\infty$, 
which is the Newton equation for a star which goes around another star with the mass $M$ 
in the de Sitter spacetime.
When we consider a star moving on the plane by $(Y^1,Y^2)$ coordinates 
(i.e., $Y^3=\cdots=Y^n=0$), 
the result (i) in (2) in Theorem \ref{Thm-4} gives the solution 
\beq
\label{Eq-Rotation-Sol}
Y^1(t)=R\cos \left(\omega t+\delta\right), 
\ \ 
Y^2(t)=R\sin \left(\omega t+\delta\right)
\eeq
for some constant $\delta$ 
if $H$ satisfies $H^2<GM/R^3$, 
where $\omega=\sqrt{GM/R^3-H^2}$.
We note that the original variable $X$ is given by $X(t)=Y(t)/a(t)=Y(t)/e^{Ht}$ 
in \eqref{X-Y}, below. 
The solution \eqref{Eq-Rotation-Sol} shows 
that the angular velocity $\omega$ of the star in the de Sitter spacetime ($H\neq0$) 
is smaller than the angular velocity $\sqrt{GM/R^3}$ in the Minkowski spacetime ($H=0$).
Let us calculate its difference by an example of the sun and the earth. 
We use the values 
\begin{eqnarray*}
&&T=\dot{T}\times 10^7 \ {\rm [s]},\ \ \dot{T}=3.1556925,\ \  
M=\dot{M}\times 10^{30} \ {\rm [kg]}, \ \ \dot{M}=1.9884,\ \ 
\\
&&R=\dot{R}\times 10^8 \ {\rm [km]}, \ \ \dot{R}=1.496,\ \ 
G=\dot{G}\times 10^{-20} \ {\rm [km^3 \ kg^{-1} \ s^{-2}] },\ \ \dot{G}=6.67408,\ \ 
\\
&& H=\dot{H}\times 10\ {\rm [km \ s^{-1}\ Mpc^{-1}] },\ \ \dot{H}=7,
\end{eqnarray*}
where ${\rm s}$ denotes the second, and {\rm Mpc} denotes Mega parsec 
which is ${\rm Mpc}=\dot{P}\times 10^{19} $ {\rm [km]} with $\dot{P}=3.085677581$.
When $H=0$, we must have 
$\omega T \left(=\sqrt{GM R^{-3}}\cdot T\right)=2\pi$,  
while we have $\sqrt{GM R^{-3}}\cdot T=2\times 3.141004674$ from the above values.
Now, we calculate the angular velocity in the de Sitter spacetime as 
\begin{eqnarray*}
\sqrt{\frac{GM}{R^3}-H^2}
&=&
\sqrt{\frac{GM}{R^3}}
\left(1-\frac{\dot{H}^2 \dot{R}^3}{\dot{G}\dot{M}\dot{P}^2}\times 10^{-22}\right)^{1/2}
\\
&\doteqdot&
\sqrt{\frac{GM}{R^3}}
\left(1-\frac{1}{2}\cdot \frac{\dot{H}^2 \dot{R}^3}{\dot{G}\dot{M}\dot{P}^2}\times 10^{-22}\right),
\end{eqnarray*}
where 
\[
\frac{\dot{H}^2 \dot{R}^3}{\dot{G}\dot{M}\dot{P}^2}\doteqdot 1.298358447.
\]
\end{remark}

\vspace{10pt}

Secondly, we consider the case $\sigma=-1$, $n=1$, $Y_0\in \br$, $Y_1\in \br$, 
$H\ge0$, $\lambda>0$, $1<p<\infty$,  
$a(t)=e^{Ht}$, $f(Y)=\lambda|Y|^p$ in \eqref{Cauchy-Y}. 
Namely, 
\beq
\label{Cauchy-Y-Variant}
\left\{
\begin{array}{l}
\displaystyle 
D_t^2 Y(t)-H^2Y(t)+\lambda |Y(t)|^p=0\ \ \mbox{for}\ t\ge0,
\\
Y(0)=Y_0,\ \ D_t Y(0)=Y_1.
\end{array}
\right.
\eeq
We say that $Y$ is the global weak solution if $Y$ satisfies 
\beq
\label{Def-Weak}
-Y_1\phi(0)+Y_0D_t\phi(0)
+\int_0^\infty Y(t)D_t^2\phi(t)-H^2Y(t) \phi(t)
+\lambda|Y(t)|^p \phi(t) dt=0
\eeq
for any $\phi\in C^2_0([0,\infty))$.
The definition of the global weak solution follows from the multiplication of $\phi$ to the differential equation in \eqref{Cauchy-Y-Variant}.
Namely,
\[
0=
(D_t^2 Y-H^2Y+\lambda |Y|^p)\phi
=
D_t(D_tY\phi-YD_t\phi)+YD_t^2\phi-H^2 Y\phi+\lambda|Y|^p\phi,
\]
which yields \eqref{Def-Weak} by the integration.
The next theorem shows the nonexistence of the nontrivial global weak solution 
of \eqref{Cauchy-Y-Variant}, 
which implies 
that nontrivial local solutions of \eqref{Cauchy-Y-Variant} must blow up in finite time.

\begin{theorem}
[Nonexistence of nontrivial global weak solution]
\label{Thm-5}
Let $n=1$, $H\ge0$, $a(t)=e^{Ht}$,  
$\lambda>0$, $1<p<\infty$.
If $HY_0+Y_1\le 0$, then the global weak solution $Y$ of \eqref{Cauchy-Y-Variant} 
must satisfy $Y=0$.
\end{theorem}

On the condition $HY_0+Y_1\le 0$ in Theorem \ref{Thm-5}, 
we note that $Y_1$ must be non-positive for the Minkowski spacetime (i.e., $H=0$), 
while $Y_1$ can be positive if $Y_0<0$ and the spatial expansion is sufficiently large 
(i.e., $H$ is sufficiently large).

\vspace{10pt}

We use the following notations.
We denote the Lebesgue space by $L^q(I)$ for an interval $I\subset \br$ 
and $1\le q\le \infty$ with the norm 
\[
\|Y\|_{L^q(I)}:=
\begin{cases}
\left\{
\int_I |Y(t)|^q dt
\right\}^{1/q} & \mbox{if}\ \ 1\le q<\infty,
\\
{\mbox{ess.}\sup}_{t\in I} |Y(t)| 
& \mbox{if}\ \ q=\infty.
\end{cases}
\]
We denote the inequality $A\le CB$ for some constant $C>0$ which is not essential for the argument by $A\lesssim B$.
Put $\nabla:=(\partial_1,\cdots,\partial_n)$ and 
$\Delta:=\sum_{j=1}^n \partial_j^2$.

This paper is organized as follows.
We show the derivation of the differential equation in \eqref{Cauchy-Y} 
based on the Ehrenfest theorem for the Schr\"odinger equation in the homogeneous and isotropic spacetime in Section \ref{Section-Derivation}. 
We also show some estimates for the solution of the linear equation, 
and the energy estimates for the inhomogeneous equation.
We prove the above theorems in Sections from \ref{Section-Thm-1} to \ref{Section-Thm-5}.

\newsection{Derivation of the equation}
\label{Section-Derivation}
In this section, we show the derivation of the first equation in \eqref{Cauchy-Y}.
We derive it from the Schr\"odinger equation 
by the Ehrenfest theorem, 
where the Schr\"odinger equation is derived from the Klein-Gordon equation by the nonrelativistic limit.

We use the following convention.
The Greek letters $\alpha, \beta,\gamma,\cdots$ run from $0$ to $n$, 
and the Latin letters $j,k,\ell, \cdots$ run from $1$ to $n$.
We use the Einstein rule for the sum of indices, 
namely, the sum is taken for same upper and lower repeated indices, 
for example, 
$\partial_j\partial^j:=\sum_{j=1}^n \partial_j\partial^j$, 
${T^\alpha}_\alpha:=\sum_{\alpha=0}^n {T^\alpha}_\alpha$ 
and 
${T^j}_j:=\sum_{j=1}^n {T^j}_j$ for any tensor ${T^\alpha}_\beta$.
We put 
$x:=(x^0,x^1,\cdots, x^n)\in \br^{1+n}$, $t:=x^0$, $x':=(x^1,\cdots, x^n)$.

Firstly, let us consider the Klein-Gordon equation 
\beq
\label{Der-1000}
\frac{1}{\sqrt{-g}} \partial_\alpha
\left(\sqrt{-g} g^{\alpha\beta}\partial_\beta \phi\right)
-\frac{m^2c^2}{\hbar^2} \phi
-\frac{2m}{\hbar^2}U(x')\phi
=0
\eeq
in the spacetime with the metric 
$(g_{\alpha\beta}(t))=\diag(-c^2,a(t)^2,\cdots, a(t)^2)$, 
where $g$ is the determinant of the matrix $(g_{\alpha\beta})$, 
the matrix $(g^{\alpha\beta})$ denotes the inverse matrix of $(g_{\alpha\beta})$, 
$x=(x^0,x^1,\cdots,x^n)$ with $x^0=t$ 
and $U=U(x')$ for $x'=(x^1,\cdots,x^n)$ is a real-valued potential.
Put 
\[
w(t):=b_0\left(\frac{a_0}{a(t)}\right)^{n/2},\ \ 
b(t):=w(t)e^{-imc^2 t/\hbar}, \ \ 
u(x):=\frac{\phi(x)}{b(t)}
\]
for the constant $a_0>0$ in \eqref{a}, and any constant $b_0(=b(0))\in \br$.

We obtain the Schr\"odinger equation by the nonrelativistic limit as follows.

\begin{lemma}
\label{Lem-4}
The Schr\"odinger equation 
\beq
\label{Der-3000}
i\frac{2m}{\hbar} \partial_tu+\frac{1}{a^2}\Delta u-\frac{2m}{\hbar^2}Uu=0
\eeq
is obtained from \eqref{Der-1000} by the nonrelativistic limit (i.e., $c \to \infty$).
The equation \eqref{Der-3000} is rewritten as 
\beq
\label{Der-4000}
i\hbar \partial_tu
-\frac{1}{2m}p^jp_j u-U u=0,
\eeq
where 
$x_\alpha:=g_{\alpha\beta}x^\beta$, 
$p_\alpha:=-i\hbar\partial_\alpha$, 
$p^\alpha:=g^{\alpha\beta}p_\beta$ 
for 
$0\le \alpha\le n$.
\end{lemma}

\begin{proof}
Since we have 
\[
\frac{1}{\sqrt{-g}} \partial_\alpha
\left(\sqrt{-g} g^{\alpha\beta}\partial_\beta \phi\right)
=
-\frac{1}{c^2}\left(
\frac{n\partial_t a}{a}\partial_t \phi+\partial_t^2\phi
\right)
+\frac{1}{a^2}\Delta \phi,
\]
the equation \eqref{Der-1000} is rewritten as 
\beq
\label{Der-2000}
-\frac{1}{c^2}\cdot I+\two=0,
\eeq
where we have put 
\[
I:=
\frac{n\partial_t a}{a}\partial_t \phi+\partial_t^2\phi
+\frac{m^2c^4}{\hbar^2}\phi,
\ \ 
\two:=
\frac{1}{a^2}\Delta \phi-\frac{2m}{\hbar^2} U(x')\phi.
\]
Put $\alpha_\ast:=- i{m}/{\hbar}$.
Since we have 
\[
I=b(\partial_t^2u+2\alpha_\ast c^2\partial_tu+u\three), 
\ \ 
\two=b\left(\frac{1}{a^2}\Delta u-\frac{2m}{\hbar^2} U u\right),
\]
where we have put 
\[
\three:=-\left(\frac{n\partial_t a}{2a}\right)^2-\frac{n}{2}\partial_t
\left(\frac{\partial_t a}{a}\right),
\] 
the equation \eqref{Der-2000} is rewritten as 
\[
0=
-\frac{1}{c^2} (\partial_t^2u+u\three)-2\alpha_\ast \partial_t u
+\frac{1}{a^2}\Delta u-\frac{2m}{\hbar^2} U u.
\]
By the nonrelativistic limit $(c\to\infty)$ of this equation, 
we obtain the Schr\"odinger equation \eqref{Der-3000}, 
which is rewritten as \eqref{Der-4000} 
by $p^j p_j=-\hbar^2 a^{-2}\Delta$.
\end{proof}

For the solution $u$ of the Schr\"odinger equation,
we define the expectation values for $x^j$, $x_j$, $p^j$ and any function $h=h(x')$ 
for $x'\in \brn$ by 
\begin{eqnarray*}
&&
\anglex{x^j}(t):=\int_\brn \overline{u(t,x')} x^j u(t,x') dx',
\ \ 
\anglex{x_j}(t):=\int_\brn \overline{u(t,x')} x_j u(t,x') dx',
\\ 
&&
\anglex{p^j}(t):=\int_\brn \overline{u(t,x')} p^j u(t,x') dx',
\ \  
\anglex{h}(t):=\int_\brn \overline{u(t,x')} h(x') u(t,x') dx'.
\end{eqnarray*}

We have the equation of motion as follows.

\begin{lemma}
\label{Lem-5}
Let $u$ be the solution of \eqref{Der-3000}. 
The equations  
\[
(1)\ \ m D_t \anglex{x^j}(t)=\anglex{p^j}(t),
\ \ 
(2)\ \ m D_t^2 \anglex{x^j}(t)=-\frac{D_t a^2(t)}{a^2(t)} \anglex{p^j}(t)-\anglex{\partial^j U}(t)
\]
hold for $t\ge0$ and $1\le j\le n$. 
Especially, the equation of motion 
\beq
\label{Lem-5-1000}
m D_t^2 \anglex{x^j}(t)+m\frac{D_t a^2(t)}{a^2(t)} D_t\anglex{x^j}(t)+\anglex{\partial^jU}(t)=0
\eeq
holds for $0\le t<T_1$.
\end{lemma}

\begin{proof}
(1) 
Put $I^j:=\overline{u} x^j i\hbar \partial_t u$ and $\two^j:=\overline{u}x^j \Delta u$.
By the equation \eqref{Der-3000}, we have 
\[
\imaginary I^j = -\frac{\hbar^2}{2m a^2}\, \imaginary \two^j.
\]

Since we have  $\nabla(\overline{u}u)=2\real (\overline{u}\nabla u)$ 
and $\int_\brn \nabla (\overline{u}u)dx'$ $=0$ by the divergence theorem, 
we have 
\beq
\label{Lem-5-1500}
\int_{\brn} \overline{u} \nabla u dx'
=
i \imaginary \int_{\brn} \overline{u} \nabla u dx'.
\eeq
Since $\two^j$ is rewritten as 
\[
\two^j=
\nabla \cdot (\overline{u} x^j\nabla u)-x^j |\nabla u|^2-\overline{u}\partial_j u,
\]
we have 
$\imaginary \two^j
=
\nabla \cdot \imaginary(\overline{u} x^j\nabla u)
-
\imaginary(\overline{u}\partial_j u)$,
which yields 
\beq
\label{Lem-5-2000}
\int_\brn \imaginary \two^j dx'
=-\imaginary \int_\brn \overline{u} \partial_j u dx'
= i\int_\brn \overline{u} \partial_j u dx'
\eeq
by \eqref{Lem-5-1500}.
By the definition of $\langle x^j \rangle$ and \eqref{Lem-5-2000}, 
we have 
\begin{eqnarray*}
i\hbar D_t \anglex{x^j}
&=& 2i\int_\brn \imaginary I^j dx' \\
&=& -\frac{i\hbar^2}{m a^2} \int_\brn \imaginary \two^j dx' \\
&=& \frac{\hbar^2 }{m a^2}\int_\brn \overline{u} \partial_j u dx'.
\end{eqnarray*}
We obtain the required equation by $p^j u=a^{-2} p_j u=-i\hbar a^{-2} \partial_j u$.

(2) 
Put $\three^j:=2\real (\partial_j \overline{u}\cdot i\hbar \partial_t u)$.
By (1) and $p^j=-i\hbar a^{-2}\partial_j$, we have 
\begin{eqnarray}
m D_t^2 \anglex{x^j}
&=& 
D_t \anglex{p^j}
\nonumber\\
&=&
-\frac{D_t a^2}{a^2}\int_\brn \overline{u} p^j u dx'
+\frac{1}{a^2}\int_\brn \three^j dx'
-\frac{1}{a^2}\int_\brn \partial_j (\overline{u} i\hbar \partial_t u) dx'
\nonumber\\
&=&
-\frac{D_t a^2}{a^2}\anglex{p^j}+\frac{1}{a^2}\int_\brn \three^j dx',
\label{Lem-5-3000}
\end{eqnarray}
where we have used the definition of $\anglex{p^j}$ and the divergence theorem.
Let us estimate the term $\three^j$. 
By the equation \eqref{Der-3000}, we have 
\[
\three^j
=
-\frac{\hbar^2}{ma^2} \real (\partial_j \overline{u} \Delta u)
+2U \real(u\partial_j \overline{u})
=: \three_1^j+\three_2^j.
\] 
By simple calculations, we have 
\begin{eqnarray*}
\three_1^j
&=& 
-\nabla\cdot
\left(
\frac{\hbar^2}{ma^2}\real (\partial_j \overline{u} \nabla u) 
\right)
+\partial_j
\left(
\frac{\hbar^2}{2ma^2}|\nabla u|^2
\right),
\\
\three_2^j
&=& 
\partial_j(U|u|^2)-\partial_j U |u|^2,
\end{eqnarray*}
which yield $\int_\brn \three_1^j dx'=0$ 
and 
$\int_\brn \three_2^j dx'=-\anglex{\partial_j U}$ by the divergence theorem 
and the definition of the expectation value $\anglex{\cdot}$.
So that, we obtain the required equation from \eqref{Lem-5-3000}.
We also obtain \eqref{Lem-5-1000} inserting the result (1) into (2). 
\end{proof}

Now, we consider the potential $U$ defined by 
\[
U(x'):=\frac{m}{2}U_\ast(\anglex{x^j}\anglex{x_j}) x^k x_k
\]
for a function $U_\ast(\cdot)$.
Put 
\[
X^j:=\anglex{x^j}, \ \ X_j:=\anglex{x_j}=a^2X^j,\ \  \mbox{and} \ \ r:=\left(X^j X_j\right)^{1/2}.
\]
Put   
\beq
\label{X-Y}
Y^j(t):=a(t)X^j(t)\ \ \mbox{for}\ \ t\ge0,
\eeq
and 
\beq
\label{Y-A}
A:=
-\frac{1}{4}\left(\frac{D_t a^2}{a^2}\right)^2
-\frac{1}{2}D_t \left(\frac{D_t a^2}{a^2}\right).
\eeq

The equation \eqref{Lem-5-1000} is rewritten as follows.

\begin{lemma}
\label{Lem-6}
The following results hold for $1\le j\le n$.

(1) $\anglex{\partial^j U}=m U_\ast(r^2) X^j$.

(2) The equation \eqref{Lem-5-1000} is rewritten as 
\beq
\label{Eq-X}
D^2_t X^j+\frac{D_t a^2}{a^2} D_t X^j+U_\ast(r^2)X^j=0.
\eeq

(3) The equation \eqref{Eq-X} is rewritten as 
\beq
\label{Lem-6-1000}
D_t^2Y^j+AY^j+U_\ast(|Y|^2)Y^j=0.
\eeq
\end{lemma}

\begin{proof}
(1) By $\partial_j(x^k x_k)=2x_j$, we have 
\[
\partial_j U(x')=\frac{m}{2}U_\ast(r^2)\partial_j(x^k x_k)=m U_\ast(r^2) x_j.
\]
Thus, we obtain 
\[
\anglex{\partial^j U(x')}
=
m U_\ast(r^2) \anglex{x^j}
=
m U_\ast(r^2) X^j.
\]

(2) 
The equation \eqref{Eq-X} follows from \eqref{Lem-5-1000} directly 
by $X^j=\anglex{x^j}$ and (1).

(3) 
We have $X^j=a^{-1} Y^j$ and $r^2=X^j X_j=|Y|^2$ 
by \eqref{Def-Modulo}.
By 
\beq
\label{Lem-6-2000}
D_t X^j=a^{-1}\left(D_t-\frac{D_t a^2}{2a^2}\right)Y^j,
\eeq
we have 
\beq
\label{Lem-6-3000}
D_t^2X^j=a^{-1}
\left[
D_t^2 Y^j-\frac{D_t a^2}{a^2} D_t Y^j 
+
\frac{1}{4}
\left\{
\left(\frac{D_t a^2}{a^2}\right)^2-2D_t\left(\frac{D_t a^2}{a^2}\right)
\right\}Y^j
\right].
\eeq
By $\partial^j U(x')=m U_\ast(r^2) x^j$, we have 
$\anglex{\partial^j U}=m U_\ast(r^2) X^j$.
We have the required equation \eqref{Lem-6-1000} by 
\eqref{Eq-X}, 
\eqref{Lem-6-2000} and \eqref{Lem-6-3000}.
\end{proof}

To rewrite the Cauchy problem \eqref{Cauchy-Y} as the integral equation,
we prepare some fundamental results for ordinary differential equations.

\begin{lemma}
\label{Lem-7}
For any fixed nonnegative function 
$\widetilde{a}\in C([0,T))$ for $T>0$, 
let $\rho_0$ and $\rho_1$ be the solutions of the Cauchy problem 
\beq
\label{Lem-7-1000}
\left\{
\begin{array}{l}
\left(D_t^2+\widetilde{a}(t)\right) \rho_j(t)=0 \ \  \mbox{for} \ \ t\in [0,T), \\
\rho_j(0)=\delta_{0j}, \ \ D_t\rho_j(0)=\delta_{1j},
\end{array}
\right.
\eeq
where $\delta_{ij}=1$ for $0\le i=j\le 1$,  $\delta_{ij}=0$ for $0\le i\neq j\le 1$.
Put 
$\widetilde{A}:=
\begin{pmatrix}
0 & 1 \\
-\widetilde{a} & 0
\end{pmatrix}$, 
\[
\Phi_m(t):=
\begin{cases}
E & \mbox{if}\ \ m=0,
\\
\int_0^t \int_0^{t_1} \cdots \int_0^{t_{m-1}} 
\widetilde{A}(t_1)
\widetilde{A}(t_2)\cdots 
\widetilde{A}(t_m) 
dt_m\cdots dt_2 dt_1 
& \mbox{if}\ \ m\ge1,
\end{cases}
\]
where $E$ denotes the unit matrix, 
$\Phi:=\sum_{m=0}^\infty \Phi_m$.
Let $b\in L^1((0,T))$, and let $\rho$ be the solution of the equation  
\beq
\label{Lem-7-2000}
(D_t^2+\widetilde{a}(t))\rho(t)=b(t)
\eeq
for $0\le t<T$.
Then the following results hold.

(1) 
$\Phi=
\begin{pmatrix}
\rho_0 & \rho_1 \\
D_t \rho_0 & D_t\rho_1
\end{pmatrix}$.

(2) $\det \Phi=1$.

(3) The solution $\rho$ is given by 
\[
\begin{pmatrix}
\rho(t) \\
D_t\rho(t)
\end{pmatrix}
=\Phi(t)
\begin{pmatrix}
\rho(0) \\
D_t\rho(0)
\end{pmatrix}
+
\int_0^t
\Phi(t)\Phi(s)^{-1} 
\begin{pmatrix}
0 \\
b(s)
\end{pmatrix}
ds,
\]
which is rewritten as 
\begin{eqnarray}
\rho(t)&=&\rho_0(t)\rho(0)+\rho_1(t)D_t \rho(0)+\int_0^t \rho_{12}(t,s) b(s) ds,
\label{Lem-7-3000}
\\
D_t\rho(t)&=&D_t\rho_0(t)\rho(0)+D_t\rho_1(t)D_t\rho(0)+\int_0^t \rho_{22}(t,s) b(s) ds,
\label{Lem-7-4000}
\end{eqnarray}
where $\rho_{12}$ and $\rho_{22}$ are defined by 
\begin{eqnarray}
\rho_{12}(t,s)&:=&-\rho_0(t)\rho_1(s)+\rho_1(t)\rho_0(s),
\label{Lem-7-4100}
\\ 
\rho_{22}(t,s)&:=&-D_t\rho_0(t)\rho_1(s)+D_t\rho_1(t)\rho_0(s).
\label{Lem-7-4200}
\end{eqnarray}

(4) 
If $\widetilde{a}\ge0$ and $D_t \widetilde{a}\le 0$ on $[0,T)$, then 
\[
|\rho_0(t)|\le 
\sqrt{ \frac{\widetilde{a}(0)}{\widetilde{a}(t)} },
\ \ 
|D_t\rho_0(t)|\le \sqrt{ \widetilde{a}(0) },
\ \ 
|\rho_1(t)|\le 
\sqrt{ \frac{1}{\widetilde{a}(t)} },
\ \ 
|D_t\rho_1(t)|\le 1.
\]

(5) 
If $\widetilde{a}\ge0$ and $D_t \widetilde{a}\ge 0$ on $[0,T)$, then 
\[
|\rho_0(t)|\le 1, 
\ \ 
|D_t\rho_0(t)|\le \sqrt{ \widetilde{a}(t) },
\ \ 
|\rho_1(t)|\le 
\sqrt{ \frac{1}{\widetilde{a}(0)} },
\ \ 
|D_t\rho_1(t)|\le 
\sqrt{ \frac{\widetilde{a}(t)}{\widetilde{a}(0)} }.
\]

(6) 
$\rho\in C([0,T))$.
Moreover, if $b\in C([0,T))$, then $\rho\in C^1([0,T))$.
\end{lemma}

\begin{proof}
(1) 
We note that the solution $\rho$ in \eqref{Lem-7-2000} with $b=0$ satisfies 
\[
D_t 
\begin{pmatrix}
\rho(t) \\
D_t\rho(t)
\end{pmatrix}
=\widetilde{A}(t) 
\begin{pmatrix}
\rho(t) \\
D_t\rho(t)
\end{pmatrix}
.
\]
We have 
$D_t \Phi=\sum_{m=1}^\infty \widetilde{A} \Phi_{m-1} =\widetilde{A}\Phi$ 
by $D_t \Phi_0=0$ and $D_t \Phi_m=\widetilde{A} \Phi_{m-1}$ for $m\ge1$. 
Thus, the solution $\rho$ in \eqref{Lem-7-2000} with $b=0$ satisfies 
\[
\begin{pmatrix}
\rho(t) \\
D_t\rho(t)
\end{pmatrix}
=
\Phi(t) 
\begin{pmatrix}
\rho(0) \\
D_t\rho(0)
\end{pmatrix}
\]
since it satisfies  
\[
D_t
\begin{pmatrix}
\rho(t) \\
D_t\rho(t)
\end{pmatrix}
=
D_t\Phi(t) 
\begin{pmatrix}
\rho(0) \\
D_t\rho(0)
\end{pmatrix}
=
\widetilde{A}(t) 
\Phi(t) 
\begin{pmatrix}
\rho(0) \\
D_t\rho(0)
\end{pmatrix}
=
\widetilde{A}(t) 
\begin{pmatrix}
\rho(t) \\
D_t\rho(t)
\end{pmatrix}.
\]
So that, the solutions $\rho_0$ and $\rho_1$ satisfy 
$
\begin{pmatrix}
\rho_0 \\
D_t\rho_0
\end{pmatrix}
=\Phi
\begin{pmatrix}
1 \\
0
\end{pmatrix}
$ 
and 
$
\begin{pmatrix}
\rho_1 \\
D_t\rho_1
\end{pmatrix}
=\Phi
\begin{pmatrix}
0 \\
1
\end{pmatrix}
$, 
which yields 
$
\begin{pmatrix}
\rho_0 & \rho_1 \\
D_t \rho_0 & D_t\rho_1
\end{pmatrix}=\Phi E=\Phi$, namely, 
the required result.

(2) 
Since $\det \Phi=\rho_0D_t\rho_1-D_t\rho_0 \rho_1$ by (1), 
we have 
\[
D_t \det \Phi
=
D_t\rho_0 D_t \rho_1
+
\rho_0D_t^2\rho_1
-
D_t^2\rho_0 \rho_1
-
D_t\rho_0D_t\rho_1=0
\]
by 
$D_t^2\rho_0=-\widetilde{a} \rho_0$ 
and 
$D_t^2\rho_1=-\widetilde{a} \rho_1$. 
So that, we have $\det \Phi(t)=\det \Phi(0)=1$ by $\Phi(0)=E$.

(3) 
Since the solution $\rho$ in \eqref{Lem-7-2000} satisfies 
\beq
\label{Proof-Lem-7-3-500}
D_t 
\begin{pmatrix}
\rho(t) \\
D_t \rho(t)
\end{pmatrix}
=\widetilde{A}(t)
\begin{pmatrix}
\rho(t) \\
D_t \rho(t)
\end{pmatrix}
+
\begin{pmatrix}
0 \\
b(t)
\end{pmatrix},
\eeq
we have
\beq
\label{Proof-Lem-7-3-1000}
\begin{pmatrix}
\rho(t) \\
D_t \rho(t)
\end{pmatrix}
=
\Phi(t)
\begin{pmatrix}
\rho(0) \\
D_t \rho(0)
\end{pmatrix}
+
\Phi(t)
\int_0^t
\Phi(s)^{-1}
\begin{pmatrix}
0 \\
b(s)
\end{pmatrix} 
ds
\eeq
since it satisfies \eqref{Proof-Lem-7-3-500}.
Since we have 
$\Phi^{-1}:=
\begin{pmatrix}
D_t \rho_1 & -\rho_1 \\
-D_t \rho_0 & \rho_0
\end{pmatrix}
$ by (2),
we obtain 
\begin{eqnarray}
&&
\Phi(t)\Phi(s)^{-1}
\nonumber\\
&=&
\begin{pmatrix}
\rho_0(t)D_t \rho_1(s)-\rho_1(t)D_t\rho_0(s) & 
-\rho_0(t)\rho_1(s)+\rho_1(t)\rho_0(s) \\
D_t\rho_0(t)D_t \rho_1(s)-D_t \rho_1(t)D_t\rho_0(s) & 
-D_t\rho_0(t)\rho_1(s)+D_t\rho_1(t)\rho_0(s) 
\end{pmatrix}
\nonumber
\\
&=:&
\begin{pmatrix}
\rho_{11}(t,s) & \rho_{12}(t,s) \\
\rho_{21}(t,s) & \rho_{22}(t,s)
\end{pmatrix},
\label{Proof-Lem-7-3-2000}
\end{eqnarray}
where we have defined $\rho_{11}$, $\rho_{12}$, $\rho_{21}$, $\rho_{22}$ by the right hand side.
We obtain 
\eqref{Lem-7-3000} 
and 
\eqref{Lem-7-4000}
by (1), 
\eqref{Proof-Lem-7-3-1000} 
and 
\eqref{Proof-Lem-7-3-2000}.

(4) 
Since $\rho_0$ and $\rho_1$ are the solutions of 
\eqref{Lem-7-1000},
we have 
\[
D_t \rho_j(D_t^2+\widetilde{a})\rho_j=0
\]
for $j=0, 1$.
By 
$D_t \rho_j D_t^2\rho_j=D_t(D_t\rho_j)^2/2$ 
and 
$D_t \rho_j \rho_j=D_t (\rho_j^2)/2$, 
we have 
\beq
\label{Lem-7-5000}
D_t (D_t \rho_j)^2+\widetilde{a} D_t \rho_j^2=0,
\eeq
which is rewritten as 
\[
D_t\left\{
(D_t \rho_j)^2+\widetilde{a}\rho_j^2
\right\}
=
D_t \widetilde{a} \cdot \rho_j^2
\]
by 
$\widetilde{a}D_t\rho_j^2=D_t\left(\widetilde{a} \rho_j^2 \right)-D_t\widetilde{a} \cdot \rho_j^2$.
Under the condition $D_t\widetilde{a}\le 0$, 
we have 
\[
(D_t \rho_j)^2(t)+\widetilde{a}(t)\rho_j^2(t)
\le 
(D_t \rho_j)^2(0)+\widetilde{a}(0)\rho_j^2(0),
\]
namely, 
\[
(D_t \rho_0)^2(t)+\widetilde{a}(t)\rho_0^2(t)
\le 
\widetilde{a}(0),
\ \ 
(D_t \rho_1)^2(t)+\widetilde{a}(t)\rho_1^2(t)
\le 
1,
\]
which yield the required inequalities.

(5) 
We rewrite \eqref{Lem-7-5000} as 
\[
\widetilde{a}^{-1} D_t(D_t\rho_j)^2+D_t \rho_j^2=0.
\]
By 
$\widetilde{a}^{-1} D_t(D_t\rho_j)^2
=
D_t\left(\widetilde{a}^{-1} (D_t\rho_j)^2\right)
+\widetilde{a}^{-2}D_t \widetilde{a} (D_t\rho_j)^2$, 
we have 
\[
D_t\left\{
\widetilde{a}^{-1} (D_t \rho_j)^2+\rho_j^2
\right\}
=
-\widetilde{a}^{-2} D_t \widetilde{a} \cdot (D_t \rho_j)^2.
\]
Under the conditions $\widetilde{a}\ge0$ and $D_t\widetilde{a}\ge 0$, 
we have 
\[
\rho_j^2(t)
+
\widetilde{a}^{-1}(t)(D_t\rho_j)^2(t)
\le 
\rho_j^2(0)
+
\left(\widetilde{a}^{-1}\left(D_t\rho_j\right)^2\right)(0),
\]
namely, 
\[
\rho_0^2(t)+\widetilde{a}^{-1}(t)(D_t\rho_0)^2(t)
\le 
1,
\ \ 
\rho_1^2(t)+\widetilde{a}^{-1}(t)(D_t\rho_1)^2(t)
\le 
\widetilde{a}^{-1}(0),
\]
which yield the required inequalities.

(6) 
We note that $\rho$ is given by \eqref{Lem-7-3000}.
Since we have $\rho_0,\rho_1\in C^2([0,T))$ by 
$\widetilde{a}\in C([0,T))$ and 
$D_t^2\rho_j=-\widetilde{a}\rho_j\in C([0,T))$,
we have 
$\rho_0(\cdot)\rho(0)+\rho_1(\cdot)D_t\rho(0)\in C^2([0,T))$ 
and 
$\rho_{12}\in C^2([0,T)\times[0,T))$.
Put 
$\phi(t):=\int_0^t\rho_{12}(t,s)b(s)ds$.
For any $t\in [0,T)$ and $\varepsilon\in \br$ with $|\varepsilon|$ sufficiently small, 
we have 
\begin{eqnarray*}
\phi(t+\varepsilon)-\phi(t)
&=&
\int_0^{t+\varepsilon}\rho_{12}(t+\varepsilon,s)b(s)ds
-
\int_0^{t}\rho_{12}(t,s)b(s)ds
\\
&=& 
\int_t^{t+\varepsilon}\rho_{12}(t+\varepsilon,s)b(s)ds
+
\int_0^{t}\left(\rho_{12}(t+\varepsilon,s)-\rho_{12}(t,s)\right)b(s)ds
\\
&=:& 
I+\two,
\end{eqnarray*}
where we take $\varepsilon>0$ when $t=0$.
Let $T_\ast$ satisfy $t+|\varepsilon|<T_\ast<T$.
We have 
\[
|I|\le \|\rho\|_{L^\infty((0,T_\ast))} \cdot \left|\int_t^{t+\varepsilon} |b(s)|ds\right|
\to 0
\]
as $\varepsilon\to0$ by $b\in L^1((0,T))$.
We also have 
\[
|\two|\le \int_0^t \left|\rho_{12}(t+\varepsilon,s)-\rho_{12}(t,s)\right|\cdot |b(s)| ds,
\]
\[ 
\left|\rho_{12}(t+\varepsilon,s)-\rho_{12}(t,s)\right|\to0
\ \ \mbox{as}\ \ \varepsilon\to0,
\] 
and 
$|\two|\le 2\|\rho_{12}\|_{L^\infty((0,T_\ast))}\cdot \|b\|_{L^1((0,T_\ast))}$,
by which we have $|\two|\to0$ as $\varepsilon\to0$ by the Lebesgue convergence theorem.
Thus, we have $\phi\in C([0,T))$.
So that, we obtain $\rho\in C([0,T))$ 
by \eqref{Lem-7-3000}.

Let us consider the case $b\in C([0,T))$. 
We have 
\[
\frac{1}{\varepsilon}\left(\phi(t+\varepsilon)-\phi(t)\right)
=\three+\four,
\]
where we have put 
\[
\three:=
\frac{1}{\varepsilon}\int_t^{t+\varepsilon} \rho_{12}(t+\varepsilon,s)b(s)ds,
\ \ 
\four:=
\int_0^t 
\frac{\rho_{12}(t+\varepsilon,s)-\rho_{12}(t,s)}{\varepsilon} \cdot b(s) ds.
\]
We have 
$\three\to \rho_{12}(t,t)b(t)=0$ 
by 
$\rho_{12}(t,t)=0$ 
and 
$b\in C([0,T))$. 
We also have 
$\four\to \int_0^t D_t \rho_{12}(t,s)b(s) ds$ 
by 
$\rho_{12}\in C^1([0,T)\times[0,T))$. 
Thus, 
we have 
\[
D_t\phi(t)=\int_0^t D_t \rho_{12}(t,s)b(s)ds.
\] 
We have $D_t\phi\in C([0,T))$ by $D_t\rho_{12}\in C([0,T)\times[0,T))$ and $b\in L^1((0,T))$.
Thus, we have $\phi\in C^1([0,T))$.
So that, we obtain 
$\rho\in C^1([0,T))$ as required if $b\in C([0,T))$.
\end{proof}

\begin{lemma}
\label{Lem-8}
The scale-function $a(\cdot)$ and $q_0(\cdot)$ defined by \eqref{a} and \eqref{A} satisfy 
the followings. 
\[
(1)\ \ 
q_0=\frac{2a_1}{a_0}
\left(1+\frac{n(1+\sigma)a_1 t}{2a_0}\right)^{-1} 
\ \ \ \ 
(2)\ \ 
D_t q_0=-\frac{n(1+\sigma)}{4} q_0^2
\]
\end{lemma}

\begin{proof}
(1) Since we have 
\[
D_t a(t)=a(t)\frac{a_1}{a_0} 
\left(
1+\frac{n(1+\sigma)a_1 t}{2a_0}
\right)^{-1}
\] 
and 
$q_0=2D_ta/a$, we obtain the required result.

(2) By (1), we have 
\[
D_t q_0
=
-\frac{n(1+\sigma)a_1^2}{a_0^2} 
\left(
1+\frac{n(1+\sigma)a_1 t}{2a_0}
\right)^{-2}
=
-\frac{n(1+\sigma)}{4}q_0^2
\] 
as required.
\end{proof}

Let us classify the cases $A>0$, $D_t A>0$, $D_t A<0$ as follows 
which are needed when we consider the energy estimates for the Cauchy problem 
\eqref{Cauchy-Y}.

\begin{lemma}
\label{Lem-9}
For the function $A$ defined by \eqref{A}, 
the following results hold.

(1) $A=\frac{n}{8}\left(\sigma+1-\frac{2}{n}\right) q_0^2$

(2) 
$A=0$ holds if and only if $\sigma=-1+2/n$ or $a_1=0$.
$A>0$ holds if and only if $\sigma>-1+2/n$ and $a_1\neq 0$. 

(3) $D_t A=-\frac{n^2}{16}\left(\sigma+1-\frac{2}{n}\right) (\sigma+1)q_0^3$

(4) 
$D_t A=0$ holds if and only if $\sigma=-1+2/n$ or $\sigma=-1$ or $a_1=0$.
$D_t A>0$ holds if and only if $(\sigma+1-2/n)(\sigma+1)a_1<0$.
$D_t A<0$ holds if and only if $(\sigma+1-2/n)(\sigma+1)a_1>0$.

(5) 
If $a_1=0$ or $\sigma=-1+2/n$, then $A=D_t A=0$.

(6) 
If $a_1>0$ and $\sigma>-1+2/n$, then $A>0$ and $D_t A<0$.

(7) 
If $a_1<0$ and $\sigma>-1+2/n$, then $A>0$ and $D_t A>0$.
\end{lemma}

\begin{proof}
The result (1) follows from the definition of $q_0$ and (2) in Lemma \ref{Lem-8} 
as 
\[
A
=
-\frac{1}{4}q_0^2-\frac{D_t q_0}{2} 
=
-\frac{q_0^2}{4}+\frac{n(1+\sigma)q_0^2}{8}
=
\frac{n q_0^2}{8}\left(\sigma+1-\frac{2}{n}\right).
\]
The result (2) follows from (1) directly, 
where $q_0=0$ is equivalent to $a_1=0$ 
by (1) in Lemma \ref{Lem-8}.
The result (3) follows from $D_t A=n(\sigma+1-2/n)q_0D_t q_0/4$ and (2) in Lemma \ref{Lem-8}.
The result (4) follows from (3), and (1) in Lemma \ref{Lem-8}.
The results (5), (6) and (7) follow from (2) and (4).
\end{proof}

\begin{lemma}[Energy estimates when $D_t A\le0$] 
\label{Lem-10}
Let $0<T\le T_1$.
For any function $h=(h^1,\cdots,h^n)$, 
let $Y=(Y^1,\cdots,Y^n)$ be the solution of the equation
\beq
\label{Lem-10-1000}
D_t^2Y^j(t)+A(t)Y^j(t)+h^j(t)=0
\eeq
for $1\le j\le n$ and $0\le t<T$.
Let $A\ge0$ and $D_t A\le 0$ on $[0,T)$.

(1) The following estimate holds;
\begin{eqnarray*}
&&\|D_t Y\|_{L^\infty((0,T))}+\|\sqrt{A} \, Y\|_{L^\infty((0,T))}+\|\sqrt{-D_tA} \, Y\|_{L^2((0,T))}
\\
&\lesssim&
|D_t Y(0)|+\sqrt{A(0)}\, |Y(0)|+\|h\|_{L^1((0,T))}.
\end{eqnarray*}

(2) Let $h:=\lambda |Y|^{p-1} Y$ for $\lambda\in \bc$ and $1<p<\infty$.
Put 
\[
e^0:=\frac{1}{2}|D_t Y|^2+\frac{1}{2}A|Y|^2+\frac{\lambda}{p+1}|Y|^{p+1}, 
\ \ 
e^1:=-\frac{1}{2} D_t A|Y|^2.
\]
Then the following estimate holds;
\[
e^0(t)+\int_0^t e^1(s)ds=e^0(0).
\]
\end{lemma}

\begin{proof}
Multiplying $D_t Y^j$ to the both sides in the equation \eqref{Lem-10-1000}, we have 
\beq
\label{Lem-10-2000}
D_t
\left\{
\frac{1}{2}|D_t Y|^2+\frac{1}{2}A|Y|^2
\right\}
-\frac{1}{2}D_t A|Y|^2+\sum_{j=1}^n h^j D_t Y^j
=0.
\eeq
Integrating the both sides in this equation on the interval $[0,t]$, we have 
\begin{eqnarray*}
&&
\frac{1}{2}|D_t Y(t)|^2+\frac{1}{2}A(t)|Y(t)|^2
+
\frac{1}{2}\left\|\sqrt{-D_t A} \, Y\right\|_{L^2((0,t))}^2
+
\sum_{j=1}^n \int_0^t h^j(s)D_t Y^j(s) ds
\\
&=&
\frac{1}{2}|D_t Y(0)|^2+\frac{1}{2}A(0)|Y(0)|^2.
\end{eqnarray*}
So that, we have 
\begin{eqnarray*}
&&
\|D_t Y\|^2_{L^\infty((0,T))}
+
\|\sqrt{A}Y\|_{L^\infty((0,T))}
+
\left\|\sqrt{-D_t A} \, Y\right\|_{L^2((0,T))}^2
\\
&\lesssim&
|D_t Y(0)|^2
+
A(0)|Y(0)|^2
+
\sum_{j=1}^n \int_0^T \left|h^j(s)D_t Y^j(s)\right| ds.
\end{eqnarray*}
Since we have 
\[
\sum_{j=1}^n \int_0^T |h^j(s)D_t Y^j(s)| ds
\le
\varepsilon\sum_{j=1}^n \|D_t Y^j\|_{L^\infty((0,T))}^2
+
\frac{1}{4\varepsilon} \sum_{j=1}^n \left(\int_0^T |h^j(s)|ds\right)^2
\]
for any $\varepsilon>0$, 
we obtain the required result taking $\varepsilon$ sufficiently small.

(2) 
Since we have 
\[
\sum_{j=1}^n h^j D_t Y^j
=
\lambda |Y|^{p-1} \sum_{j=1}^n Y^j D_t Y^j
=
\frac{\lambda}{p+1} D_t |Y|^{p+1},
\]
we obtain the required result by \eqref{Lem-10-2000}.
\end{proof}

\begin{lemma}[Energy estimates when $D_t A\ge0$]
\label{Lem-12}
For any function $h=(h^1,\cdots,h^n)$, 
let $Y=(Y^1,\cdots,Y^n)$ be the solution of the equation
\beq
\label{Lem-12-1000}
D_t^2Y^j(t)+A(t)Y^j(t)+h^j(t)=0
\eeq
for $1\le j\le n$ and $0\le t<T$.
Let $A>0$ and $D_t A\ge 0$ on $[0,T)$.

(1) The following estimate holds;
\begin{eqnarray*}
&&
\|A^{-1/2} D_t Y\|_{L^\infty((0,T))}
+\|Y\|_{L^\infty((0,T))}
+\|A^{-1}\sqrt{D_t A} \, D_t Y\|_{L^2((0,T))}
\\
&\lesssim&
|A(0)^{-1/2} D_t Y(0)|+|Y(0)|+\|A^{-1/2} h\|_{L^1((0,T))}.
\end{eqnarray*}

(2) Let $h:=\lambda |Y|^{p-1} Y$ for $\lambda\in \bc$ and $1<p<\infty$.
Put 
\begin{eqnarray*}
e^0&:=&\frac{1}{2}|A^{-1/2}D_t Y|^2+\frac{1}{2}|Y|^2+\frac{\lambda}{p+1}A^{-1}|Y|^{p+1}, 
\\ 
e^1&:=&\frac{1}{2} A^{-2} D_t A|D_t Y|^2+\frac{\lambda}{p+1} A^{-2} D_t A |Y|^{p+1}.
\end{eqnarray*}
Then the following estimate holds;
\[
e^0(t)+\int_0^t e^1(s)ds=e^0(0).
\]
\end{lemma}

\begin{proof}
(1) 
Multiplying $A^{-1}D_t Y^j$ to the both sides in the equation \eqref{Lem-12-1000}, we have 
\beq
\label{Lem-12-2000}
D_t
\left\{
\frac{1}{2}A^{-1}|D_t Y|^2+\frac{1}{2}|Y|^2
\right\}
+
\frac{1}{2}A^{-2} D_t A|D_t Y|^2+A^{-1} \sum_{j=1}^n h^j D_t Y^j
=0.
\eeq
Integrating the both sides in this equation on the interval $[0,t]$, we have 
\begin{multline*}
\frac{1}{2}A^{-1}(t)|D_t Y(t)|^2+\frac{1}{2}|Y(t)|^2
+
\frac{1}{2}\left\|A^{-1}\sqrt{D_t A} \, D_t Y\right\|_{L^2((0,t))}^2
\\
+
\sum_{j=1}^n \int_0^t A^{-1}(s)h^j(s)D_t Y^j(s) ds
=
\frac{1}{2}A^{-1}(0)|D_t Y(0)|^2+\frac{1}{2}|Y(0)|^2.
\end{multline*}
So that, we have 
\begin{multline*}
\|A^{-1/2}D_t Y\|_{L^\infty((0,T))}^2
+
\|Y\|_{L^\infty((0,T))}^2
+
\left\|A^{-1}\sqrt{D_t A} \, D_t Y\right\|_{L^2((0,T))}^2
\\
\lesssim
A^{-1}(0)|D_t Y(0)|^2
+
|Y(0)|^2
+
\sum_{j=1}^n \int_0^T \left|A^{-1}(s)h^j(s)D_t Y^j(s)\right| ds.
\end{multline*}
Since we have 
\begin{multline*}
\sum_{j=1}^n \int_0^T \left|A^{-1}(s) h^j(s)D_t Y^j(s)\right| ds
\le
\|A^{-1/2}D_t Y\|_{L^\infty((0,T))}\|A^{-1/2}h\|_{L^1((0,T))}
\\
\le
\varepsilon \|A^{-1/2} D_t Y\|_{L^\infty((0,T))}^2
+
\frac{1}{4\varepsilon} 
\|A^{-1/2} h\|_{L^1((0,T))}^2
\end{multline*}
for any $\varepsilon>0$, 
we obtain the required result taking $\varepsilon$ sufficiently small.

(2) 
Since we have 
\begin{eqnarray*}
 A^{-1}\sum_{j=1}^n h^j D_t Y^j
&=&
\lambda A^{-1} |Y|^{p-1} \sum_{j=1}^n Y^j D_t Y^j
\\
&=&
\frac{\lambda}{p+1} A^{-1}D_t |Y|^{p+1}
\\
&=&
D_t \left(\frac{\lambda}{p+1} A^{-1}|Y|^{p+1}\right)
+
\frac{\lambda}{p+1} A^{-2}D_t A|Y|^{p+1},
\end{eqnarray*}
we obtain the required result by \eqref{Lem-12-2000}.
\end{proof}

We prepare some estimates for the semilinear term in 
\eqref{Def-f(Y)} as follows.

\begin{lemma}
\label{Lem-11}
For any $Y$, $Z\in \brn$ and $1<p<\infty$, the following inequalities hold.

(1) $||Y|^{p-1}Y-|Z|^{p-1}Z|\le p\left(|Y|^{p-1}+|Z|^{p-1}\right)|Y-Z|$

(2) $||Y|^{p}-|Z|^p|\le p\left( \max\{|Y|,|Z|\} \right)^{p-1} |Y-Z|$
\end{lemma}

\begin{proof}
When $Y=0$ or $Z=0$ or $Y=Z$, the results are trivial.
We assume $Y\neq0$, $Z\neq0$ and $Y\neq Z$.

(1) 
If there is not $\theta$ such that $0<\theta<1$ and $Z+\theta(Y-Z)=0$,
then we have 
\begin{eqnarray*}
&& |Y|^{p-1}Y-|Z|^{p-1}Z
\\
&=&
\int_0^1 \frac{d}{d\theta} 
\left(
|Z+\theta(Y-Z)|^{p-1} 
\cdot (Z+\theta(Y-Z))
\right) 
d\theta
\\
&=&
(p-1)\int_0^1
|Z+\theta(Y-Z)|^{p-3}(Z+\theta(Y-Z))\cdot (Y-Z) 
(Z+\theta(Y-Z))
d\theta
\\
&& 
+
\int_0^1
|Z+\theta(Y-Z)|^{p-1}(Y-Z) d\theta.
\end{eqnarray*}
Thus, we have 
\begin{eqnarray*}
\left|
|Y|^{p-1}Y-|Z|^{p-1}Z
\right|
&\le& 
p\int_0^1 |Z+\theta(Y-Z)|^{p-1}|Y-Z| d\theta 
\\
&\le& 
p\max\{ |Y|^{p-1}, |Z|^{p-1} \} \, |Y-Z|.
\end{eqnarray*}
Let us consider the case 
that there is $\theta$ such that $0<\theta <1$ and $Z+\theta(Y-Z)=0$.
For this $\theta$, we put $\alpha:=\theta/(1-\theta)>0$. 
Since we have $Z=-\alpha Y$ and $Y-Z=(1+\alpha)Y$, 
we have 
\[
|Y|^{p-1}Y-|Z|^{p-1}Z
=
|Y|^{p-1}Y(1+\alpha^p)
=
|Y|^{p-1}
\frac{1+\alpha^p}{1+\alpha}(Y-Z).
\]
Since we have $1+\alpha^p\le (1+\alpha^{p-1})(1+\alpha)$, 
we have 
\[
\left|
|Y|^{p-1}Y-|Z|^{p-1}Z
\right|
\le 
|Y|^{p-1}(1+\alpha^{p-1})
|Y-Z|
=
(|Y|+|Z|)^{p-1}|Y-Z|.
\]
So that, we obtain the required result.

(2) 
If there is not $\theta$ such that $0<\theta<1$ and $Z+\theta(Y-Z)=0$,
then we have 
\begin{eqnarray*}
|Y|^{p}-|Z|^p
&=&
\int_0^1 \frac{d}{d\theta} \left(|Z+\theta(Y-Z)|^2\right)^{p/2} d\theta
\\
&=&
p\int_0^1 |Z+\theta(Y-Z)|^{p-2} 
(Z+\theta(Y-Z))\cdot (Y-Z) d\theta.
\end{eqnarray*}
Thus, we have 
\[
\left|
|Y|^{p}-|Z|^p
\right|
\le
p\left( \max\{|Y|,|Z|\} \right)^{p-1} |Y-Z|.
\]
Let us consider the case 
that there is $\theta$ such that $0<\theta <1$ and $Z+\theta(Y-Z)=0$.
For this $\theta$, we put $\alpha:=\theta/(1-\theta)>0$. 
We have $Z=-\alpha Y$ and 
\[
|Y|^{p}-|Z|^p
=
|Y|^p(1-\alpha^p)
\]
and 
\[
1-\alpha^p
=
\int_0^1 \frac{d}{d\tau}(\alpha+\tau(1-\alpha))^p d\tau
=
p\int_0^1 (\alpha+\tau(1-\alpha))^{p-1} d\tau \cdot (1-\alpha).
\]
Thus, we have 
\[
\left||Y|^{p}-|Z|^p\right|
\le 
p |Y|^p \left(\max\{1,\alpha\}\right)^{p-1} |1-\alpha|.
\]
Since we have $|Y|\max\{1,\alpha\}=\max\{|Y|,|Z|\}$ and 
$|Y| \cdot|1-\alpha|=||Y|-|Z||\le|Y-Z|$, 
we obtain the required result.
\end{proof}

\newsection{Proof of Theorem \ref{Thm-1}}
\label{Section-Thm-1}
We have $A>0$, $D_t A<0$ and $T_0<0$ 
by the assumption $a_1>0$ and $\sigma>-1+2/n$.
We prove Theorem \ref{Thm-1} only for the case 
$f(Y):=\lambda |Y|^{p-1}Y$ 
since the case $f(Y):=\lambda |Y|^p$ is proved analogously 
by the use of (2) instead of (1) in Lemma \ref{Lem-11}.
Put 
\beq
\label{Def-U-H}
U_\ast(|Y|^2):=\lambda|Y|^{p-1}.
\eeq
We regard the solution of the Cauchy problem \eqref{Cauchy-Y} 
as the fixed point of the operator $\Psi$ defined by 
\beq
\label{Def-Psi}
\Psi(Y)^j:=
\rho_0(t)Y_0^j+\rho_1(t) Y_1^j
-
\int_0^t \rho_{12}(t,s) f(Y)^j(s) ds
\eeq
for $1\le j\le n$, 
where $\rho_0$, $\rho_1$ and $\rho_{12}$ are the functions 
in Lemma \ref{Lem-7} 
with $\widetilde{a}(t):=A(t)$.

(1) 
Put $A_q:=A^{1/2-1/q} (-D_t A)^{1/q}$ for $2\le q\le \infty$.
By the H\"older inequality, we have 
\beq
\|A_q Y\|_{L^q((0,T))}
\le \|\sqrt{A}\, Y\|_\infty^{1-2/q} \|\sqrt{-D_t A}\, Y\|_2^{2/q}.
\eeq
Since we have $A=n(\sigma+1-2/n)q_0^2/8$ 
and 
$-D_t A=n^2(\sigma+1-2/n)(\sigma+1)q_0^3/16$ 
by (1) and (3) in Lemma \ref{Lem-9}, 
we have 
\beq
\label{Proof-Thm-1-100}
A_q=C q_0^{1+1/q}
\eeq 
for some constant $C>0$,
where $q_0>0$ by $a_1>0$ and (1) in Lemma \ref{Lem-8}.
Let $q_\ast$ be a real number which satisfies 
\beq
\label{Proof-Thm-1-500}
\max\left\{0,1-\frac{p}{2}\right\}\le \frac{1}{q_\ast}\le 1.
\eeq
Then there exist $q_1$ and $q_2$ with $2\le q_1,q_2\le \infty$ and 
\beq
\label{Proof-Thm-1-550}
\frac{1}{q_\ast}=1-\frac{p-1}{q_1}-\frac{1}{q_2}.
\eeq
Especially, \eqref{Proof-Thm-1-500} is satisfied for $q_\ast=\infty$ and $p\ge2$.
By $f(Y)=\lambda|Y|^{p-1}Y$, we have 
\beq
\label{Proof-Thm-1-300}
\|f(Y)\|_{L^1((0,T))}
\le |\lambda| I \|A_{q_1}Y\|_{L^{q_1}((0,T))}^{p-1} \|A_{q_2}Y\|_{L^{q_2}((0,T))}
\le |\lambda| I \|Y\|_{X}^p
\eeq
if $q_\ast$ satisfies $1\le q_\ast\le \infty$ and \eqref{Proof-Thm-1-550},
where we have put 
\beq
\label{Proof-Thm-1-I}
I:=\|A_{q_1}^{-p+1}A_{q_2}^{-1}\|_{L^{q_\ast}((0,T))}.
\eeq

By \eqref{Proof-Thm-1-100}, we have
\beq
\label{Proof-Thm-1-700}
I\lesssim \|q_0^{-p-1+1/q_\ast}\|_{L^{q_\ast}((0,T))}=:\two.
\eeq
By Lemma \ref{Lem-10}, we have 
\[
\|\Psi(Y)\|_X\lesssim |D_t Y(0)|+\sqrt{A(0)}|Y(0)|+\|f(Y)\|_{L^1((0,T))},
\]
which yields 
\[
\|\Psi(Y)\|_X\le C_0 D+C\two R^p
\]
for any $Y\in X(T,R)$ 
by \eqref{Proof-Thm-1-300} and \eqref{Proof-Thm-1-700},
where $C_0>0$ and $C>0$ are constants independent of $Y$.
Thus, the operator $\Psi$ maps $X(T,R)$ into itself if $R>0$ satisfies 
\beq
\label{Proof-Thm-1-1000}
R\ge 2C_0D, 
\ \ 
2C\two R^{p-1}\le 1.
\eeq

Put  
$C':=
\left({2a_1}/{a_0}\right)^{-p-1+1/q_\ast}
\left\{
{1}/{(p+1)q_\ast}
\right\}^{1/q_\ast}$.
We have $q_0=2a_1/a_0\left(1-t/T_0\right)$ 
by (1) in Lemma \ref{Lem-8}.
Since we have 
\[
\two=
C'\cdot
\left[
-T_0
\left\{\left(1-\frac{T}{T_0}\right)^{(p+1)q_\ast}-1\right\}
\right]^{1/q_\ast}
\]
when $q_\ast\neq\infty$, 
the condition $2C\two R^{p-1}\le 1$ is rewritten as 
\beq
\label{Proof-Thm-1-2000}
T\le -T_0
\left[
\left\{
1-\frac{1}{T_0(2C C' R^{p-1})^{q_\ast} }
\right\}^{1/(p+1)q_\ast}
-1
\right].
\eeq
Since we have $\two=C'\left(1-{T}/{T_0}\right)^{p+1}$ when $q_\ast=\infty$, 
the condition $2C\two R^{p-1}\le 1$ is rewritten as 
\beq
\label{Proof-Thm-1-3000}
T\le -T_0
\left\{
(2C C' R^{p-1})^{-1/(p+1)} -1
\right\} 
\ \ 
\mbox{and}
\ \ R\le \left(\frac{1}{2CC'}\right)^{1/(p-1)}
\eeq
when $q_\ast=\infty$.

Since we have 
$f(Y)-f(Z)=\lambda(|Y|^{p-1}Y-|Z|^{p-1}Z)$, 
we have
$$
|f(Y)-f(Z)|\le|\lambda|p(|Y|^{p-1}+|Z|^{p-1})|Y-Z|
$$ 
by Lemma \ref{Lem-11}. 
We have 
\begin{eqnarray}
&& \|f(Y)-f(Z)\|_{L^1((0,T))}
\nonumber\\
&\lesssim& |\lambda| I 
\left(
\|A_{q_1}Y\|_{L^{q_1}((0,T))}
+
\|A_{q_1}Z\|_{L^{q_1}((0,T))}
\right)^{p-1} \|A_{q_2}(Y-Z)\|_{L^{q_2}((0,T))}
\nonumber\\
&\lesssim& 
|\lambda| I 
\max\{ \|Y\|_X, \|Z\|_X \}^{p-1} \|Y-Z\|_X
\label{Proof-Thm-1-4000}
\end{eqnarray}
by the similar argument to derive \eqref{Proof-Thm-1-300}.
Since we have 
\[
\|\Psi(Y)-\Psi(Z)\|_X
\lesssim
\|f(Y)-f(Z)\|_{L^1((0,T))}
\]
by Lemma \ref{Lem-10},
there exist a constant $C>0$ such that  
\beq
\label{Proof-Thm-1-5000}
\|\Psi(Y)-\Psi(Z)\|_X
\le
CIR^{p-1} \|Y-Z\|_X
\eeq
holds for any $Y,Z\in X(T,R)$. 
Thus, $\Psi$ is a contraction mapping on $X(T,R)$ under the conditions 
\eqref{Proof-Thm-1-1000}.
By the Banach fixed point theorem, $\Psi$ has a unique fixed point $Y\in X(T,R)$.

(2) 
Let $Y$ be the fixed point of $\Psi$ obtained in (1).
Since $A\in C([0,T))$, and $f(Y)\in L^1((0,T))$ 
by \eqref{Proof-Thm-1-300} and $Y\in X(T)$, 
we have $Y\in C([0,T))$ by (6) in Lemma \ref{Lem-7}.
Thus, we have $f(Y)=\lambda|Y|^{p-1}Y\in C([0,T))$, 
which yields $Y\in C^1([0,T))$ by (6) in Lemma \ref{Lem-7}.

(3) 
Let $Y$ be the solution in $C^1([0,T))\cap X(T)$ of \eqref{Cauchy-Y}. 
Let $Z$ be another solution in $C^1([0,T))\cap X(T)$ of \eqref{Cauchy-Y} 
with $Z(0)=Y_0$ and $D_t Z(0)=Y_1$.
Put $T_\ast:=\sup\{t\in [0,T);\ Y(t)=Z(t)\}$.
We show $T_\ast=T$ by which the uniqueness of the solution $Y=Z$ follows. 
Assume $T_\ast<T$.
Let $\varepsilon>0$ be a small number such that $T_\ast+\varepsilon<T$.
Put the interval $J:=(T_\ast, T_\ast+\varepsilon)$.
Let $\|\cdot\|_{X(J)}$ be the norm defined by 
\eqref{Def-X-Norm} with the interval $(0,T)$ replaced by $J$.  
By the analogous argument on 
\eqref{Proof-Thm-1-4000} and \eqref{Proof-Thm-1-5000}, 
we have 
\begin{multline*}
\|Y-Z\|_{X(J)}
\le
C_0\left\{
|D_t(Y-Z)(T_\ast)|+\sqrt{A(T_\ast)} |(Y-Z)(T_\ast)|
\right\}
\\
+
C|\lambda| I(J)  
\max\{ \|Y\|_{X(J)}, \|Z\|_{X(J)} \}^{p-1} \|Y-Z\|_{X(J)},
\end{multline*}
where $I(J)$ is defined by \eqref{Proof-Thm-1-I} with the interval $(0,T)$ 
replaced by $J$.
Since 
$|D_t(Y-Z)(T_\ast)|=|(Y-Z)(T_\ast)|=0$ 
and  
$C|\lambda| I(J)  
\max\{ \|Y\|_{X(J)}, \|Z\|_{X(J)} \}^{p-1}<1$ 
for sufficiently small $\varepsilon>0$ 
by the continuity of 
$Y,Z\in C^1([0,T))$, 
\eqref{Proof-Thm-1-1000},
$\|Y\|_{X(J)}\le R$ 
and $\|Z\|_{X(J)} \to \|Y\|_{X(J)}$ as $\varepsilon \searrow0$,
we obtain $\|Y-Z\|_{X(J)}=0$ for sufficiently small $\varepsilon>0$.
Thus, we have $Y=Z$ on $[T_\ast, T_\ast+\varepsilon)$, 
which contradicts to the definition of $T_\ast$.
So that, we obtain $T_\ast=T$ as required.

(4) 
By the energy estimate Lemma \ref{Lem-10}, we have 
\[
e^0(t)+\int_0^t e^1(s) ds=e^0(0)
\] 
for $0\le t<T$.
Since $D_t A\le 0$ and $\lambda\ge0$, we have $\int_0^t e^1(s)ds\ge0$ and 
\[
\frac{1}{2} |D_t Y(t)|^2+\frac{1}{2} A(t)|Y(t)|^2\le e^0(t)\le e^0(0).
\]
So that, since $\sqrt{A(t)} |Y(t)|+|D_t Y(t)|$ is uniformly bounded for $t\in [0,T)$, 
we can obtain the global solution on $[0,T_1)$ connecting the local solution 
on short intervals in $[0,T_1)$
since the time interval for the existence of the local solution starting from $t$ 
with $0\le t<T_1$ can be uniformly taken dependent on $e^0(0)$ by the estimate 
\eqref{Thm-1-1000}.

\newsection{Proof of Theorem \ref{Thm-2}}
\label{Section-Thm-2}
We note $A(\cdot)=0$ holds when $a_1=0$ and $\sigma\in \br$, 
or $a_1\in \br$ and $\sigma=-1+2/n$ by Lemma \ref{Lem-9}.

(1) 
Let us consider the operator $\Psi$ defined by \eqref{Def-Psi}.
By the elementary equation 
$\Psi(Y)(t)=Y_0+\int_0^t D_t \Psi(Y)(\tau)d\tau$, 
we have 
\beq
\label{Proof-Thm-2-1000}
\|\Psi(Y)\|_{L^\infty((0,T))}
\le |Y_0|+T\|D_t\Psi(Y)\|_{L^\infty((0,T))} 
\le |Y_0|+TR_1
\eeq
for $Y\in X(T,R_0,R_1)$.
We have 
\beq
\label{Proof-Thm-2-1500}
\|f(Y)\|_{L^1((0,T))}
\le T\|f(Y)\|_{L^\infty((0,T))} 
\le |\lambda|T\|Y\|_{L^\infty((0,T))}^p\le |\lambda|TR_0^p
\eeq
and 
\begin{eqnarray}
\|f(Y)-f(Z)\|_{L^1((0,T))}
&\le& 
T\|f(Y)-f(Z)\|_{L^\infty((0,T))} 
\nonumber\\
&\le& 
|\lambda|p T
\left(\|Y\|_{L^\infty((0,T))}^{p-1}+\|Z\|_{L^\infty((0,T))}^{p-1}\right) \|Y-Z\|_{L^\infty((0,T))}
\nonumber\\
&\le& |\lambda|p TR_0^{p-1}\|Y-Z\|_X
\label{Proof-Thm-2-1700}
\end{eqnarray}
by Lemma \ref{Lem-11}.
We have 
\beq
\label{Proof-Thm-2-2000}
\|D_t\Psi(Y)\|_{L^\infty((0,T))} 
\lesssim |Y_1|+\|f(Y)\|_{L^1((0,T))}
\lesssim |Y_1|+|\lambda| TR_0^p
\eeq
by (1) in Lemma \ref{Lem-10}, $A(\cdot)=0$ and \eqref{Proof-Thm-2-1500}.
Since $\Psi(Y)$ and $\Psi(Z)$ satisfy 
\[
D_t^2\left(\Psi(Y)-\Psi(Z)\right)+f(Y)-f(Z)=0,
\]
we have 
\beq
\label{Proof-Thm-2-3000}
\|D_t(\Psi(Y)-\Psi(Z))\|_{L^\infty((0,T))}
\lesssim
\|f(Y)-f(Z)\|_{L^1((0,T))}
\le |\lambda| TR_0^{p-1}\|Y-Z\|_X
\eeq
similarly to \eqref{Proof-Thm-2-2000} 
by \eqref{Proof-Thm-2-1700}.
Moreover, 
by 
$\left(\Psi(Y)-\Psi(Z)\right)(t)=
\int_0^t D_t(\Psi(Y)-\Psi(Z))(\tau)d\tau$, 
we have 
\begin{eqnarray}
\|\Psi(Y)-\Psi(Z)\|_{L^\infty((0,T))}
&\le& 
T\|D_t(\Psi(Y)-\Psi(Z))\|_{L^\infty((0,T))}
\nonumber\\
&\lesssim& 
|\lambda| T^2 R_0^{p-1} \|Y-Z\|_X.
\label{Proof-Thm-2-4000}
\end{eqnarray}
By  
\eqref{Proof-Thm-2-1000},
\eqref{Proof-Thm-2-2000}, 
\eqref{Proof-Thm-2-3000} and 
\eqref{Proof-Thm-2-4000}, 
we have
\[
\|\Psi(Y)\|_{L^\infty((0,T))} 
\le |Y_0|+TR_1 
\le R_0,
\]
\[
\|D_t\Psi(Y)\|_{L^\infty((0,T))} 
\le C_0|Y_1|+CTR_0^p 
\le R_1,
\]
\[
\|D_t(\Psi(Y)-\Psi(Z))\|_{L^\infty((0,T))} 
\le CTR_0^{p-1}\|Y-Z\|_X 
\le \frac{1}{2}\|Y-Z\|_X,
\]
and 
\[
\|\Psi(Y)-\Psi(Z)\|_{L^\infty((0,T))} 
\le CT^2 R_0^{p-1}\|Y-Z\|_X 
\le \frac{1}{2}\|Y-Z\|_X
\]
for some constants $C_0>0$ and $C>0$ if 
\begin{eqnarray*}
&&R_0\ge 2|Y_0|,\ \ 2TR_1\le R_0,\ \ R_1\ge 2C_0|Y_1|, 
\\
&&2CTR_0^p\le R_1, \ \ 2CTR_0^{p-1}\le 1, \ \ 2CT^2R_0^{p-1}\le 1.
\end{eqnarray*}
Since these conditions are satisfied under the condition 
\eqref{Thm-2-T}, 
$\Psi$ is a contraction mapping on $X(T,R_0,R_1)$.
The solution is obtained as the fixed point of $\Psi$.

(2) 
When $a_1<0$ and $\sigma=-1+2/n$, 
we have $T_1=T_0>0$.
By the argument in (1), we obtain the global solution 
if \eqref{Thm-2-T} holds 
with $T$ replaced by $T_0$, 
which is satisfied if $R_0=2T_0R_1$ and $R_0$ is sufficiently small.
So that, we obtain the global solution if $|Y_0|$ and $|Y_1|$ are sufficiently small.

The results of (2), (3) and (4) in Theorem \ref{Thm-1} follow from the similar proofs for 
Theorem \ref{Thm-1}. 
Especially for (4), the energy estimate 
\[
\frac{1}{2}|D_t Y(t)|^2
+
\frac{\lambda}{p+1}|Y(t)|^{p+1}
=
\frac{1}{2}|Y_1|^2
+
\frac{\lambda}{p+1}|Y_0|^{p+1}
\]
by (2) in Lemma \ref{Lem-10} 
shows the boundedness of $Y(t)$ and $D_t Y(t)$ for $0\le t<T_1$ 
when $\lambda>0$, 
by which we obtain the global solution.
When $\lambda=0$, we also have the global solution 
since the differential equation in \eqref{Cauchy-Y} is linear for $Y$.

\newsection{Proof of Theorem \ref{Thm-3}}
\label{Section-Thm-3}
We note $A>0$, $D_t A>0$ and $T_0>0$ holds 
when $a_1<0$ and $\sigma>-1+2/n$ by Lemma \ref{Lem-9}.

(1) 
Let us consider the operator $\Psi$ defined by \eqref{Def-Psi}.
We have 
\[
\|A^{-1/2}f(Y)\|_{L^1((0,T))}
\le |\lambda| \|A^{-1/2}\|_{L^1((0,T))} \|Y\|_{L^\infty((0,T))}^p.
\]
Since $A=n(\sigma+1-2/n)q_0^2/8$ by (1) in Lemma \ref{Lem-9}, 
we have 
\[
\|A^{-1/2}\|_{L^1((0,T))}
\lesssim
\||q_0|^{-1}\|_{L^1((0,T))}
=
\frac{a_0}{2|a_1|}
\left\|1-\frac{t}{T_0}\right\|_{L^1((0,T))}
= 
\frac{a_0}{2|a_1|} T\left(1-\frac{T}{2T_0}\right)
\]
by Lemma \ref{Lem-8}.
Since we have 
\[
\|\Psi(Y)\|_{X'(T)}
\lesssim D'+\|A^{-1/2}f(Y)\|_{L^1((0,T))}
\]
by Lemma \ref{Lem-12}, 
there exist constants $C_0>0$ and $C>0$ such that  
\[
\|\Psi(Y)\|_{X'(T)}
\le C_0 D'+\frac{C|\lambda|a_0}{|a_1|}T\left(1-\frac{T}{2T_0}\right)R^p
\]
holds for any $Y\in X'(T,R)$,
where $D'$ is defined by \eqref{Def-D'}.
Thus, we have 
$\|\Psi(Y)\|_{X'(T)}\le R$ if $T$ and $R$ satisfy 
\beq
\label{Proof-Thm-3-1000}
R\ge 2C_0D',
\ \ 
\frac{2C|\lambda|a_0}{|a_1|}T\left(1-\frac{T}{2T_0}\right)R^{p-1}\le 1.
\eeq
Similarly to \eqref{Proof-Thm-1-4000} and \eqref{Proof-Thm-1-5000}, 
we are able to show 
\[
\|\Psi(Y)-\Psi(Z)\|_{X'(T)}\le \frac{1}{2}\|Y-Z\|_{X'(T)}
\]
for any $Y,Z\in X'(T,R)$ under the conditions \eqref{Proof-Thm-3-1000}.
So that, $\Psi$ is a contraction mapping on $X'(T,R)$, 
and the solution is obtained as its fixed point.

(2) 
Since $T(1-T/2T_0)\le T\le T_0$ by $T_0>0$, 
the second condition in \eqref{Proof-Thm-3-1000} is satisfied if 
\[
\frac{2C|\lambda|a_0}{|a_1|}T_0R^{p-1}
\,
\left(
=
\frac{4C|\lambda|}{n(1+\sigma)H^2}R^{p-1}
\right)\le 1.
\]
Thus, the conditions \eqref{Proof-Thm-3-1000} are satisfied with $T=T_0$ 
if $R>0$ and $D'>0$ are sufficiently small.
Namely, we obtain the global solution for small data.

(3) 
For the solution $Y\in X'(T)$,  
since $f(Y)\in L^1((0,T))$ by 
\[
\|f(Y)\|_{L^1((0,T))}\le T\|f(Y)\|_{L^\infty((0,T))}
\] 
and 
\[
\|f(Y)\|_{L^\infty((0,T))}
\le |\lambda|\|Y\|_{L^\infty((0,T))}^p 
\le |\lambda|\|Y\|_{X'(T)}^p ,
\]
we have $Y\in C([0,T))$ by (6) in Lemma \ref{Lem-7}.
Thus, we have $f(Y)\in C([0,T))$, 
and moreover $Y\in C^1([0,T))$ again 
by (6) in Lemma \ref{Lem-7}.

(4) The result follows from the analogous argument in the proof of (3) in Theorem \ref{Thm-1}.

(5) 
By the energy estimate Lemma \ref{Lem-12}, we have 
$e^0(t)+\int_0^t e^1(s) ds=e^0(0)$ for $0\le t< T\le T_0$.
Since $D_t A >0$ and $\lambda\ge0$, 
we have $\int_0^t e^1(s)ds\ge0$ and 
\[
\frac{1}{2} |A^{-1/2}(t) D_t Y(t)|^2+\frac{1}{2} |Y(t)|^2
\le e^0(t)\le e^0(0).
\]
Since $D'(t):=|Y(t)|+A(t)^{-1/2} |D_t Y(t)|$ is uniformly bounded 
on $[0,T)$, 
the existence time $T$ for the local solution can be taken uniformly 
under the condition  
\eqref{Thm-3-1000} 
with $a_0$, $a_1$ and $D'$ replaced by $a(t)$, $D_ta(t)$ and $D'(t)$.
So that, we can obtain the global solution on $[0,T_0)$ connecting the local solutions  
on short intervals in $[0,T_0)$.

\newsection{Proof of Theorem \ref{Thm-4}}
\label{Section-Thm-4}
We have $A=-H^2$ by $a(t)=e^{Ht}$ and \eqref{A}.

(1) 
Since the differential equation in \eqref{Cauchy-Y} is rewritten as 
$D_t^2Y+(\lambda-H^2)Y=0$ when $p=1$, 
the solution $Y$ is given by \eqref{Thm-4-1000} with \eqref{Thm-4-2000}.

(2) 
The differential equation in \eqref{Cauchy-Y} is rewritten as 
\beq
\label{Proof-Thm-4-1000}
D_t^2Y+(\lambda R^{p-1}-H^2)Y=0
\eeq
by $R=|Y|$.
Let $Y^3=\cdots=Y^n=0$.

(i) 
When $R$ satisfies $\lambda R^{p-1}-H^2>0$, 
then $Y$ is given by \eqref{Thm-4-3000}.
Here, $\lambda R^{p-1}-H^2>0$ holds if and only if 
$\lambda>0$ and  $R>(H^2/\lambda)^{1/(p-1)}$.

(ii) 
Let $R$ satisfy $\lambda R^{p-1}-H^2=0$, 
which holds if and only if 
$\lambda=H=0$, or $\lambda>0$ and $R=(H^2/\lambda)^{1/(p-1)}$.
Since \eqref{Proof-Thm-4-1000} is rewritten as 
$D_t^2Y=0$, 
$Y$ is given by $Y^j=B^j+C^jt$ for some constants $B^j$ and $C^j$ 
for $j=1,2$.
Since $(B^1)^2+(B^2)^2=R^2$ and $C^1=C^2=0$ by $|Y|=R$, 
we obtain $Y=(B^1,B^2,0,\cdots,0)$ with 
$B^1=Y^1_0$, $B^2=Y^2_0$, $Y^1_1=C^1=0$, $Y^2_1=C^2=0$ 
and 
$(Y^1_0)^2+(Y^2_0)^2=R^2$.

(iii) 
Let $R$ satisfy $\lambda R^{p-1}-H^2<0$.
Since \eqref{Proof-Thm-4-1000} is rewritten as 
$D_t^2Y-(H^2-\lambda R^{p-1})Y=0$, 
$Y$ is given by 
\[
Y^j=B^j e^{ \sqrt{ H^2-\lambda R^{p-1} }t }
+
C^j e^{ -\sqrt{ H^2-\lambda R^{p-1} }t }
\]
for some constants $B^j$ and $C^j$ for $j=1,2$.
We have $B^1=B^2=0$ 
by $|Y|\to\infty$ as $t\to\infty$ 
if $B^1\neq0$ or $B^2\neq0$  
which contradicts to $|Y|=R$.
Moreover, $R=0$ must hold by 
$|Y|=\sqrt{(C^1)^2+(C^2)^2}e^{-\sqrt{H^2-\lambda R^{p-1}}t}\to0$ as $t\to\infty$.
Thus, $Y=0$ is only allowed.

\newsection{Proof of Theorem \ref{Thm-5}}
\label{Section-Thm-5}
Since we have 
\[
|Y(t)|=e^{Ht} |X(t)|
\ \ \mbox{and}\ \ 
A=-H^2
\]
by \eqref{X-Y} and \eqref{Y-A}, 
where $|X|:=\left\{\sum_{j=1}^n (X^j)^2\right\}^{1/2}$ by \eqref{Def-Modulo},
the Cauchy problem \eqref{Cauchy-Y-Variant} 
is rewritten as 
\beq
\label{Cauchy-X-Variant}
\begin{cases}
D_t^2X(t)+2H D_t X+\lambda e^{(p-1)H t} |X(t)|^p=0 & \mbox{for}\ \ t\ge0,
\\
X(0)=X_0,
\ \ 
D_t X(0)=X_1
\end{cases}
\eeq
when $a(t)=e^{Ht}$ by Lemma \ref{Lem-6}.
We say that $X$ is a global weak solution if $X$ satisfies 
\begin{multline}
\label{Def-Weak-X}
-X_1\phi(0)+X_0D_t\phi(0)-2H X_0\phi(0)
\\
+\int_0^\infty 
X(t)D_t^2\phi-2H X(t) D_t\phi(t)
+\lambda e^{(p-1)Ht}|X(t)|^p \phi(t) 
dt=0
\end{multline}
for any $\phi\in C^2_0([0,\infty))$ 
which is equivalent to \eqref{Def-Weak}.
So that, Theorem \ref{Thm-5} is equivalent to the following Theorem \ref{Thm-5-X} 
since 
$X_0=Y_0$ and $X_1=-HY_0+Y_1$ 
when $a(t)=e^{Ht}$ under \eqref{X-Y}.
It suffices to show Theorem \ref{Thm-5-X} to prove Theorem \ref{Thm-5}.

\begin{theorem}
\label{Thm-5-X}
Let $n=1$, $\sigma=-1$ in \eqref{a}, $H\ge0$, $\lambda>0$, $1<p<\infty$.
If $X_1+2H X_0\le 0$, 
then any global weak solution $X$ of \eqref{Cauchy-X-Variant} 
must satisfy $X=0$.
\end{theorem}

\begin{proof}
Let $\eta\in C_0^\infty([0,\infty))$ be a non-negative function 
with 
$\eta(t)=1$ for $0\le t\le 1/2$, 
$\eta(t)=0$ for $t\ge1$, 
and $|D_t\eta(t)|^2/\eta(t) \lesssim 1$ for $1/2\le t\le 1$.
An example of $\eta$ is given by 
\[
\eta(t)
:=
\begin{cases}
1 & \mbox{if}\ \ 0\le t\le \frac{1}{2}, 
\\
1-\eta_0\int_{1/2}^t e^{-(s-1/2)^{-1}(1-s)^{-1} } ds
& \mbox{if}\ \ \frac{1}{2}<t<1, 
\\
0 & \mbox{if}\ \ t\ge1, 
\end{cases}
\]
where $\eta_0:=\left(\int_{1/2}^1 e^{-(s-1/2)^{-1}(1-s)^{-1} } ds\right)^{-1}$.
For $R>0$, put 
$\eta_R(t):=\eta(t/R)$ 
and 
$\phi:=\eta_R^{p'}$,
where $p'$ is defined by $1/p+1/p'=1$. 
We note $\phi\in C^2([0,\infty))$
by 
\begin{eqnarray*}
D_t \eta_R^{p'}(t)
&=&\frac{p'}{R} \eta_R^{p'-1}(t)D_t\eta\left(\frac{t}{R}\right),
\\
D_t^2 \eta_R^{p'}(t)
&=&
\frac{p'(p'-1)}{R^2} \eta_R^{p'-1}(t)
\cdot 
\frac{\left(D_t\eta(t/R)\right)^2}{\eta(t/R)}
+
\frac{p'}{R^2}\eta_R^{p'-1}(t)D_t^2\eta\left(\frac{t}{R}\right)
\end{eqnarray*}
and $|D_t\eta(t)|^2/\eta(t) \lesssim 1$ for $1/2\le t\le 1$.
The equation \eqref{Def-Weak-X} is rewritten as 
\beq
\label{IJL}
\lambda I=X_1+2H X_0-J+L
\eeq
by $\phi(0)=1$ and $D_t\phi(0)=0$,
where we have put 
\[
\begin{array}{l}
I:=\int_0^\infty e^{(p-1)Ht} |X(t)|^p\phi(t) dt,
\ \ 
J:=\int_0^\infty X(t) D_t^2 \phi(t) dt,
\\ 
L:=2H\int_0^\infty X(t)D_t\phi(t) dt.
\end{array}
\]
By $|D_t^2\eta_R^{p'}|\lesssim R^{-2}\eta_R^{p'-1} \chi_{[R/2,R]}$ 
and the H\"older inequality, 
we have 
\begin{eqnarray*}
|J|
&\lesssim& 
\frac{1}{R^2} \int_{R/2}^R |X(t)|\eta_R^{p'-1}(t) dt
\\
&\lesssim&
\frac{1}{R^2} 
\left\{\int_{R/2}^R e^{-Ht}  dt\right\}^{1/p'} 
\left\{\int_{R/2}^R e^{(p-1)Ht} |X(t)|^p \eta_R^{p'}(t) dt\right\}^{1/p} 
\\
&\lesssim&
\frac{J_\ast I^{1/p}}{R^2},
\end{eqnarray*}
where 
$\chi_{[R/2,R]}$ is the characteristic function on the interval $[R/2,R]$, 
and 
we have put 
\beq
\label{Proof-Thm-5-X-1000}
J_\ast:=\left\{\int_{R/2}^R e^{-Ht} dt\right\}^{1/p'}
\left(
\lesssim 
R^{1/p'}\max\left\{e^{-HR/2p'}, \ e^{-HR/p'} \right\}
\right).
\eeq
By $|D_t\eta_R^{p'}|\lesssim R^{-1}\eta_R^{p'-1} \chi_{[R/2,R]}$ 
and the H\"older inequality, 
we have 
\begin{eqnarray*}
|L|
&\lesssim& 
\frac{H}{R} \int_{R/2}^R |X(t)|\eta_R^{p'-1}(t) dt
\\
&\lesssim&
\frac{H}{R} 
\left\{\int_{R/2}^R e^{-Ht}  dt\right\}^{1/p'} 
\left\{\int_{R/2}^R e^{(p-1)Ht} |X(t)|^p \eta_R^{p'}(t) dt\right\}^{1/p} 
\\
&\lesssim&
\frac{H J_\ast I^{1/p}}{R}.
\end{eqnarray*}
Since we have $X_1+2H X_0\le 0$ in \eqref{IJL} by the assumption,
we obtain 
\[
\lambda I \le -J+L\le |J|+|L|
\lesssim 
\frac{J_\ast I^{1/p}}{R^2}
+
\frac{H J_\ast I^{1/p}}{R}.
\]
Dividing the both sides by $I^{1/p}$, we have 
\[
\lambda I^{1/p'} 
\lesssim 
\frac{J_\ast }{R^2}
+
\frac{H J_\ast }{R}
\le
\left(R^{1/p'-2}+HR^{1/p'-1}\right)
\max\left\{e^{-HR/2p'}, \ e^{-HR/p'} \right\}
\]
by \eqref{Proof-Thm-5-X-1000}.
Since $\lambda>0$ and the right hand side tends to $0$ as $R$ tends to infinity by $H\ge0$, 
we have 
$\lim_{R\to\infty} I=0$, which yields 
\[
\int_0^\infty e^{(p-1)Ht} |X(t)|^p dt=0
\]
by the definition of $I$.
So that, we obtain $X=0$ as required.
\end{proof}

\vspace{10pt}

%%%%%%%%%%%%%%%%%%%%%%%%%%%%%%%%%%
%
%%%%%%%%%%%%%%%%%%%%%%%%%%%%%%%%%%

{\bf Acknowledgment.}
This work was supported by JSPS KAKENHI Grant Number 16H03940.
%The author is thankful to the anonymous referee for several comments to revise the paper.

%%%%%%%%%%%%%%%%%%%%%%%%%%%%%%%%%%%%%
%
%%%%%%%%%%%%%%%%%%%%%%%%%%%%%%%%%%%%%


{\small 
\begin{thebibliography}{99}

%\bibitem{Agemi-2000-InventMath}
%R. Agemi, 
%%Agemi, Rentaro
%\emph{Global existence of nonlinear elastic waves}, 
%%. (English) Zbl 1095.35009 
%Invent. Math. {\bf 142} (2000), No. 2, 225-250. 

\bibitem{Astashova-2009-FunctDiffEq}
I. V. Astashova, 
%Astashova, I. V. 
\emph{On asymptotical behavior of solutions to a quasi-linear second order differential equation}, 
Funct. Differ. Equ.  {\bf 16}  (2009),  no. 1, 93--115.

\bibitem{Astashova-2019-DiffEq}
 I. V. Astashova,
%Astashova, I. V. 
\emph{Asymptotic behavior of singular solutions of Emden-Fowler type equations}, 
Translation of Differ. Uravn. {\bf 55} (2019), no. 5, 597--606. Differ. Equ.  {\bf 55}  (2019),  no. 5, 581--590. 

\bibitem{Baskin-2013-AHP}
D. Baskin,
%Dean Baskin
\emph{Strichartz Estimates on Asymptotically de Sitter Spaces},
Annales Henri Poincar{\'e} {\bf 14} (2013), Issue 2, pp 221--252.

\bibitem{Carroll-2004-Addison}
S. Carroll, 
%Carroll, Sean(1-CHI)
\emph{Spacetime and geometry.  
An introduction to general relativity}, 
Addison Wesley, San Francisco, CA, 2004, xiv+513 pp.

\bibitem{DInverno-1992-Oxford}
R. d'Inverno, 
%d'Inverno, Ray(4-SHMP)
\emph{Introducing Einstein's relativity}, 
The Clarendon Press, Oxford University Press, New York, 1992, xii+383 pp. 

\bibitem{Knezhevich-2007-DiffEq}
Y. Knezhevich-Milyanovich, 
%Knezhevich-Milyanovich, Yu.(SE-BELG)
\emph{Vertical asymptotes of solutions of the Emden-Fowler equation}, 
%(Russian. Russian summary) 
Differ. Uravn.  {\bf 43}  (2007),  no. 12, 1710--1711;  
%1728;  
translation in 
Differ. Equ.  {\bf 43}  (2007),  no. 12, 1753--1755. 

\bibitem{Knezhevich-2009-DiffEq}
Y. Knezhevich-Milyanovich,  
%Knezhevich-Milyanovich, Yu. 
\emph{On the Cauchy problem for an equation of Emden-Fowler type}, 
% (Russian) 
Differ. Uravn.  {\bf 45}  (2009),  no. 2, 260--262;  
translation in  Differ. Equ.  {\bf 45}  (2009),  no. 2, 267--270.

\bibitem{Krtinic-Mikic-2019-MathNotes}
D. Krtinich, M. Mikich, 
%Krtinich, Dzh.; Mikich, M. 
\emph{On the Cauchy problem for a generalized Emden-Fowler-type equation}, 
% (Russian) 
Mat. Zametki  {\bf 105}  (2019),  no. 1, 153--157;  
translation in  Math. Notes  {\bf 105}  (2019),  no. 1--2, 148--152. 

\bibitem{Kwong-Wong-2006-NA}
M. K. Kwong, J. S. W. Wong,  
%Kwong, Man Kam; Wong, James S. W. 
A nonoscillation theorem for sublinear Emden-Fowler equations, 
Nonlinear Anal.  {\bf 64}  (2006),  no. 7, 1641--1646. 

\bibitem{Li-2006-NA}
M-R. Li,  
%Li, Meng-Rong 
\emph{On the Emden-Fowler equation $u''-|u|^{p-1}u=0$}, 
Nonlinear Anal.  {\bf 64}  (2006),  no. 5, 1025--1056.

\bibitem{Mikic-2016-KragujevacJMath}
M. Miki\'c,  
%Mikic, Marija 
\emph{Note about asymptotic behaviour of positive solutions of superlinear differential equation of Emden-Fowler type at zero}, 
Kragujevac J. Math.  {\bf 40}  (2016),  no. 1, 105--112. 

\bibitem{Nakamura-2014-JMAA}
M. Nakamura, 
%Nakamura, Makoto
\emph{The Cauchy problem for semi-linear Klein-Gordon equations in de Sitter spacetime},
J. Math. Anal. Appl.  {\bf 410}  (2014),  no. 1, 445--454. 

\bibitem{Nakamura-2015-JDE}
M. Nakamura,
%Makoto Nakamura
\emph{On nonlinear Schr\"odinger equations derived from the nonrelativistic limit of nonlinear Klein-Gordon equations in de Sitter spacetime}, 
Journal of Differential Equations {\bf 259} (2015), 3366--3388.

\bibitem{Nakamura-9999-KJM}
M. Nakamura,
%Makoto Nakamura
\emph{On the nonrelativistic limit of a semilinear field equation in a uniform and isotropic space}, 
Kyoto Journal of Mathematics (in press).

\bibitem{Nakamura-2017-ComplexEinstein}
M. Nakamura, 
\emph{Remarks on the derivation of several second order 
partial differential equations from a generalization 
of the Einstein equations}, 
Osaka Journal of Mathematics (in press).

%\bibitem{Nakamura-9999-NS-EW}
%M. Nakamura, 
%\emph{Remarks on the Navier-Stokes equations 
%and the elastic wave equations 
%in homogeneous and isotropic spacetimes}, preprint.

\bibitem{Nakamura-Sato-9999-Diffusion}
M. Nakamura, Y. Sato,
\emph{Remarks on global solutions for the semilinear diffusion equation in the de Sitter spacetime}, 
Hokkaido Mathematical Journal (in press).

\bibitem{Yagdjian-2012-JMAA}
K. Yagdjian,
%Yagdjian, Karen 
\emph{Global existence of the scalar field in de Sitter spacetime},
J. Math. Anal. Appl. {\bf 396} (2012), no. 1, 323--344. 

\end{thebibliography}
}
\end{document}